\documentclass[10pt]{article}
\usepackage{hyperref}
\hypersetup{
  colorlinks=true, 
  urlcolor=blue, 
  linkcolor=blue, 
  citecolor=red    
}
\usepackage{fancyhdr}
\usepackage{textcomp}
\usepackage[mathscr]{eucal}
\usepackage{amsmath,amsfonts,amsthm,amscd,amssymb,scalefnt}
\usepackage[margin=1in]{geometry}    
\usepackage{graphicx}                
\usepackage{subfig}
\usepackage{caption}
\usepackage{float}
\usepackage{epstopdf}
\usepackage{multirow}
\usepackage[shortlabels]{enumitem}
\numberwithin{equation}{section}
\usepackage{cite}

\DeclareMathOperator{\sign}{sign}

\newtheorem{theorem}{Theorem}{}
\newtheorem{corollary}{Corollary}{}
\newtheorem{definition}{Definition}{}
\newtheorem{remark}{Remark}{}
{}
\newtheorem{proposition}{Proposition}{}
{}

\numberwithin{definition}{section}
\numberwithin{theorem}{section}
\numberwithin{corollary}{section}
\numberwithin{remark}{section}
\numberwithin{lemma}{section}
\numberwithin{proposition}{section}
\numberwithin{equation}{section}
\numberwithin{example}{section}
\newcommand{\thedate}{\today}
\makeatother

\begin{document}

\begin{center}
\Large \bf{Natural and Conjugate Mates of Frenet Curves in Three-Dimensional Lie Group}
\end{center}
\vspace{12pt}
\centerline{\large Mahmut MAK\footnote[1]{ E-mail: mmak@ahievran.edu.tr\hfill\thedate}}
\vspace{12pt}

\centerline{$^{1}$\it K{\i}r\c{s}ehir Ahi Evran University, Faculty of Arts and Sciences, Department of Mathematics, K{\i}r\c{s}ehir, Turkey.}
\vspace{12pt}

\begin{abstract}
In this study, we introduce the natural mate and conjugate mate of a Frenet curve in a three dimensional Lie group $ \mathbb{G} $ with bi-invariant metric. Also, we give some relationships between a Frenet curve and its natural mate or its conjugate mate in  $ \mathbb{G} $. Especially, we obtain some results for the natural mate and the conjugate mate of a Frenet curve in $ \mathbb{G} $ when the Frenet curve is a general helix, a slant helix, a spherical curve, a rectifying curve, a Salkowski (constant curvature and non-constant torsion), anti-Salkowski (non-constant curvature and constant torsion), Bertrand curve. Finally, we give nice graphics with numeric solution in Euclidean 3-space as a commutative Lie group.
\end{abstract}

$~~~${\bf Keywords:} Natural mate,  conjugate mate, helix, slant helix, spherical curve, rectifying curve,\\
\indent{~~~Salkowski curve, anti-Salkowski curve.}\smallskip

$~~~${\bf MSC2010:}$~~$53A04, 53A35, 22E15.

\section{Introduction}
In the theory of curves in differential geometry, "generating a new curve from a regular curve, examining the relationships between them and obtaining new characterizations for them" is always a matter of curiosity and has taken its place among popular topics. In this sense, in the category of curves associated with the Frenet vector fields of regular curves; Bertrand curve, involute-evolute curve, Mannheim curve, principal-direction or binormal-direction curve, and also with respect to their position vector; rectifying curve, osculating curve, normal curve are among the leading examples. These curves and their geometric properties have been studied by many authors in different ambient spaces.

From the fundamental theorem of curves, we have general information about the structure of the curve if the curvatures of any regular curve are known. Therefore, in the theory of curves, "to give the characterization of the curve in terms of curvatures" is another attracted topic. For example; in Euclidean 3-space, for any regular curve with curvature $\kappa $ and torsion $\tau $, the following characterizations are well known:
\begin{itemize}
\item $\kappa =0$ iff it is a straight line,
\item $\tau =0$ iff it is a planar curve,
\item $\tau =0$ and $\kappa = c$ is a non-zero constant iff it is a circle with radius $c^{-1} $ \cite{Carmo},
\end{itemize}
and under assuming that  $\kappa \ne 0$; 
\begin{itemize}
\item $ {\left( {{{\left( {1/\kappa} \right)}^\prime }(1/\tau)} \right)^2} + {\left( {1/\kappa} \right)^2} = {r^2} $ iff it is a spherical curve which is lying on a sphere with radius $ r $ \cite{Carmo},
\item the ratio $(\tau /\kappa )$ is a constant iff it is a general helix. Especially, both of the curvatures are a non-zero constant iff it is a circular helix \cite{Lancret1806,Struik1961},
\item the ratio $(\tau /\kappa) $ is a linear function with respect to arc-length parameter iff it is rectifying curve \cite{chen2003},
\item $ \left( {{\kappa^2}/{{\left( {{\tau^2} + {\kappa^2}} \right)}^{3/2}}} \right){\left( {\tau/\kappa} \right)^\prime } $ is a constant iff it is a slant helix \cite{izumiya},
\item $\kappa $ is a constant but $\tau $ is a non-constant function iff it is a Salkowski curve. Conversely,  $\kappa $ is a non-constant function but $\tau $ is a constant iff it is a anti-Salkowski curve \cite{salkowski,monterde}.   
\end{itemize}

The other versions of the associated curves and the special curves were also defined in various ambient spaces such as Riemannian or Lorentzian space forms. In particularly, these curves and their geometric properties are also studied in three-dimensional Lie groups. In this sense, general helix \cite{ciftci}, slant helix \cite{okuyucu}, rectifying curve \cite{bozkurt} and recently Darboux helix \cite{ozturk} as special curves; Bertrand curve \cite{okuyucu_2016}, Mannheim curve \cite{gok} as mate curves; and also principal-direction curve \cite{kiziltug} as associated curves, are studied in a three-dimensional Lie group with bi-invariant Riemannian metric.

In Euclidean 3-space, Choi and Kim introduced the concept of principal direction curve and binormal direction curve, are defined as the integral curve of principal normal $ N $  and binormal $ B $ of a Frenet curve (i.e. $ \kappa \neq 0 $), respectively \cite{choi}. However, in Euclidean 3-space, Deshmukh et al.\cite{deshmukh} defined a new mate curve which is tangent to the principal normal vector (resp. binormal vector) of the curve, and it is called the natural mate curve (resp. conjugate mate curve). The natural mate curve and conjugate mate curve are same with principal-direction curve and binormal-direction curve from algebraic viewpoint, respectively. But, Desmukh et al. \cite{deshmukh} used the terminology of natural mate or conjugate mate, which is more accurate and comprehensive from geometric viewpoint since the integral curve is defined only for vector fields on a region which containing a curve, not along a curve. This idea be also valid in three-dimensional Lie groups.

In this paper, by using this idea, we defined natural mate and conjugate mate of a Frenet curve in a three-dimensional Lie group $\mathbb{G}$ with bi-invariant Riemannian metric. Moreover, we give some relationships between a Frenet curve and its natural mate. Especially, we obtain new characterizations for the natural mate of a Frenet curve which is a general helix, a slant helix, a spherical curve, a rectifying curve and a curve with constant curvature or torsion in $\mathbb{G}$. However, we get some corollaries for the conjugate mate of a Frenet curve which is a general helix, slant helix in  $ \mathbb{G} $. Finally, we show that a Frenet curve and its conjugate mate curve are Bertrand mate curves and involute-evolute curves.   
\section{Preliminary}
Let $ \mathbb{G} $ be a three dimensional Lie Group with a bi-invariant Riemannian metric $ < , > $ and $\mathfrak{g}$ be the Lie algebra of $\mathbb{G}$, which is consisted of all smooth vector fields of $\mathbb{G}$. Then $\mathfrak{g}$ is isomorphic to $T_e{\mathbb{G}}$, where $e$ is identity of $\mathbb{G}$. Moreover the following equations 
\begin{eqnarray} \label{eq_2.1}
\left\langle {X,[Y,Z]} \right\rangle-\left\langle {[X,Y],Z} \right\rangle=0
\end{eqnarray}
and
\begin{eqnarray} \label{eq_2.2}
\nabla_{X}Y = \dfrac{1}{2}[X,Y]
\end{eqnarray}
are satisfied with respect to bi-invariant metric for all $X, Y,Z \in \mathfrak{g}$, where $\nabla$ is the Levi-Civita connection of Lie group $\mathbb{G}$.

Let $ \gamma:I\subset{\mathbb{R}} \to \mathbb{G} $ be arc-lenghted (unit speed) curve  and $ \{X_{1},X_{2},X_{3}\}$ be an orthonormal basis of $\mathfrak{g}$. $ U = \sum_{i=1}^{3}{{u_{i}}{X_{i}}} $ and $ V = \sum_{j=1}^{3}{{v_{j}}{X_{j}}} $ along the curve $\gamma$ where $u_{i}$ and $v_{j}$ are smooth functions from $ I $ to $\mathbb{R}$. Moreover, Lie bracket of two vector fields $U$ and $ V $ along the curve $\gamma$ is given by 
\begin{eqnarray} \label{eq_2.3}
[U,V] = \sum_{i,j=1}^{3}{{u_{i}}{v_{j}}}[X_{i},X_{j} ].
\end{eqnarray}
Let $\nabla_{T}{U}$ is the covariant derivative of $U$ along the curve $\gamma$ and it is given by 
\begin{eqnarray} \label{eq_2.4}
\nabla_{T} U = U' + \dfrac{1}{2}[T,U] 
\end{eqnarray} 
where $T = \gamma'$ and $U' = \sum_{i=1}^{3}{{u_{i}'}{X_{i}}}$ such that $u_{i}'(t) = \dfrac{du_{i}}{dt} , t\in I$. Note that, if $U$ is the restriction of a left-invariant vector field to $\gamma$, then $U' = 0$.

Let $ \gamma:I\subset{\mathbb{R}} \to \mathbb{G} $ be a unit speed curve with the Frenet-Serret apparatus $ \{T,N,B,\kappa,\tau \}$ such that $\kappa>0$. Then the Frenet-Serret equations of $\gamma$ is given by 
\begin{eqnarray} \label{eq_2.5}
\nabla_{T}{T} = \kappa{N},\, \nabla_{T}{N} = -\kappa{T}+\tau{B},\, \nabla_{T}{B} = - \tau{N},
\end{eqnarray}
where the curvature and torsion of $\gamma$ is 
\begin{eqnarray} \label{eq_2.6}
\kappa = \left\| \nabla_{T}{T} \right\| ,~ \tau = \left\langle \nabla_{T}N,{B} \right\rangle.
\end{eqnarray}
If torsion $ \tau \neq 0$, the curve $\gamma$ is called a Frenet curve in $ \mathbb{G} $. Also, a smooth function $ \tau_\mathbb{G} $ which is called Lie group torsion of $ \gamma $, is given by    
\begin{eqnarray} \label{eq_2.7}
\tau_\mathbb{G}= \dfrac{1}{2}{\left\langle [T,N],B \right\rangle}.
\end{eqnarray}

\begin{proposition}\label{prp2_01}
Let $\gamma:I\subset{\mathbb{R}} \to \mathbb{G}$ be a unit speed curve in $\mathbb{G}$ with Frenet-Serret apparatus $ \{T,N,B,\kappa,\tau \}$. Then the following equations	
\begin{eqnarray} \label{eq_2.8}
\left\{ \begin{array}{l}
[T,N] = \left\langle {[T,N],B} \right\rangle B = 2{\tau_\mathbb{G}}B,\\
{[N,B]} = \left\langle {[N,B],T} \right\rangle T = 2{\tau_\mathbb{G}}T,\\
{[B,T]} = \left\langle {[T,B],N} \right\rangle N =  - 2{\tau_\mathbb{G}}N.
\end{array} \right.
\end{eqnarray}
are satisfied \cite{ciftci,yoon}.
\end{proposition}

Let $ \tau_\mathbb{G} $ be the Lie group torsion of $ \gamma $, which is given by (\ref{eq_2.7}). Then, we easily see that 
\begin{eqnarray} \label{eq_2.9}
\left\{ \begin{array}{l}
{T}'= {\kappa}{N}, \\
{N}'=- {\kappa}{T} + (\tau - \tau_\mathbb{G}){B},\\
{B}'= - (\tau - \tau_\mathbb{G}){N},
\end{array} \right.
\end{eqnarray}
by using the equations (\ref{eq_2.4}), (\ref{eq_2.5}), (\ref{eq_2.7}) and (\ref{eq_2.8}).


\begin{definition} \label{defn_2.2}
Let $\gamma$ be a unit speed curve in $\mathbb{G}$ with Frenet-Serret apparatus $ \{T,N,B,\kappa,\tau \}$. Then, the harmonic curvature function of the curve $\gamma$ is given by 
\begin{eqnarray} \label{eq_2.10}
H = \dfrac{\tau - \tau_\mathbb{G}}{\kappa}
\end{eqnarray}
where $\tau_\mathbb{G}$ is the Lie group torsion of $ \gamma $ \cite{okuyucu}.
\end{definition}

\begin{theorem} \label{teo_2.1}
The curve $\gamma$ is a general helix in $\mathbb{G}$ iff its harmonic function is a constant function \cite{ciftci,okuyucu}.
\end{theorem} 


\begin{theorem} \label{teo_2.2}
The curve $\gamma$ is a slant helix in $\mathbb{G}$  iff the function 
\begin{eqnarray} \label{eq_2.11}
\sigma = \dfrac{\kappa(H^2+1)^{3/2}}{H'}
\end{eqnarray}
is a constant function, where $H$ is harmonic curvature of $\gamma$ \cite{okuyucu}.
\end{theorem}


\begin{theorem} \label{teo_2.3}
	Let $\gamma$ be a unit speed curve in $\mathbb{G}$ with Frenet-Serret apparatus $ \{T,N,B,\kappa,\tau\}$. Then $\gamma$ is a rectifying curve iff the harmonic function $H$ is linear function of are length parameter $s$ of $\gamma$ (i.e. $H(s) = as + b$, where $a \neq 0$ and b are constant) \cite{bozkurt}.
\end{theorem}

\begin{definition} \label{defn_2.4}
Let $\gamma$ be a unit speed curve with the Frenet-Serret apparatus $ \{T,N,B,\kappa,\tau \}$. Then the Darboux vector of the curve $\gamma$ is given by 
\begin{eqnarray} \label{eq_2.12}
{D} = \tau{T} + \kappa{B},
\end{eqnarray}
which is satisfying the following equations  
\begin{eqnarray} \label{eq_2.13}
\nabla_{T}{T} = {{D}} \times {T} ,~\nabla_{T}{N} = {{D}} \times {N},~ \nabla_{T}{B} = {{D}} \times {B}.
\end{eqnarray}
where "$ \times $" is (natural) cross product in three-dimensional Lie algebra $ \mathfrak{g} $. Moreover, we define the vector field is given by 
\begin{eqnarray} \label{eq_2.14}
\Omega = (\tau - \tau_\mathbb{G}){T} + {\kappa}{B},
\end{eqnarray}
which is satisfying the following equations
\begin{eqnarray} \label{eq_2.15}
T' = {\Omega} \times {T} ,~ N' = {\Omega} \times {N} , ~B' = {\Omega} \times {B}.
\end{eqnarray}
Hence $\Omega$ is called extrinsic Darboux vector of the curve $\gamma$ with respect to usual derivative and its length is 
\begin{eqnarray} \label{eq_2.16}
{\omega} = \sqrt{(\tau - \tau_\mathbb{G})^2 + {\kappa}^2}.
\end{eqnarray}
\end{definition}

\begin{definition} \label{defn_2.8}
The extrinsic co-Darboux vector of the curve $\gamma$ is given by 
\begin{eqnarray} \label{eq_2.17}
\Omega^* = -{\kappa}{T} + (\tau - \tau_\mathbb{G}){B}.
\end{eqnarray}
It is easily seen that $\Omega^*$ corresponds to usual derivative ($\,\,'$) of the principal normal of the curve $\gamma$. That is, $\Omega^* = N'$.
\end{definition}

\begin{definition} \label{defn_2.9}
Let $\gamma:I\subset{\mathbb{R}} \to \mathbb{G}$ be an arc length parametrized curve, then a curve $\alpha:I\subset{\mathbb{R}} \to \mathfrak{g}$ where $\mathfrak{g}$ is the Lie algebra of $\mathbb{G}$, for which ${\alpha}'(s) = dL_{{\gamma}^{-1}(s)}{\gamma}'(s)$ for all $s \in I$ is called the left shift of $\gamma$ \cite{ciftci}.
\end{definition}

\begin{remark} \label{rmrk_2.01}
A left shift is a canonical map from the tangent space of $\mathbb{G}$ to the Lie algebra $\mathfrak{g}$ \cite{fokas}.
\end{remark}

\begin{proposition} \label{prp_2.01}
Let $\gamma:I\subset{\mathbb{R}} \to \mathbb{G}$ be a curve, then there exists a left shift (unique up to initial conditions) $\alpha:I\subset{\mathbb{R}} \to \mathfrak{g}$ \cite{ciftci}.
\end{proposition}

\begin{definition} \label{defn_2.10}
The curve $\gamma$ is called a spherical curve in $ \mathbb{G} $ if the left shift $\alpha$ of $ \gamma $ lies on the unit central sphere in $\mathfrak{g}$ (i.e. $\left\langle \alpha(t),\alpha(t) \right\rangle = 1$ for all $t \in I$) \cite{ciftci}.
\end{definition}

By using equation (\ref{eq_2.9}) and Definition \ref{defn_2.10}, we get easily the following characterization of spherical curves in $ \mathbb{G}$. 

\begin{theorem} \label{teo_2.4}
 Let $\gamma$ be a unit speed Frenet curve in $\mathbb{G}$ with Frenet-Serret apparatus $ \{T,N,B,\kappa,\tau \}$. Then $\gamma$ is a spherical curve whose left shift is lying a sphere with radius $ r $ if and only if the following equations satisfy:
 \begin{enumerate}[label={\upshape(\roman*)}, align=left, widest=i, leftmargin=5pt]\vspace{-5pt}
 	\item if $ {\tau-{\tau_\mathbb{G}}} = 0$, then $ \kappa=1/r $,
 	\item if $ {\tau-{\tau_\mathbb{G}}} \neq 0$, then
 \begin{eqnarray}\label{eq_2.18}
 \left({\left(\dfrac{1}{\kappa}\right)'{\dfrac{1}{\tau - \tau_\mathbb{G}}}}\right)^2 + \left(\dfrac{1}{\kappa}\right)^2 = r^2,
\end{eqnarray}
or equivalently,
\begin{eqnarray}\label{eq_2.19}
{\left( {{{\left( {\frac{1}{\kappa }} \right)}^\prime }\frac{1}{{\tau  - {\tau _\mathbb{G}}}}} \right)^\prime } + H = 0.
\end{eqnarray}
where $H$ is harmonic curvature of $\gamma$.
\end{enumerate}\vspace{-5pt}

\end{theorem}

\begin{remark} \label{rmrk_2.2}
Under special cases, it is known that a three dimensional Lie group $\mathbb{G}$ with bi-invariant metric contains $ \mathbb{S}^1\times \mathbb{S}^1 \times \mathbb{S}^1$, $ \mathbb{S}^1\times \mathbb{S}^1 \times \mathbb{R}$, $ \mathbb{S}^1 \times \mathbb{R}^2 $, $ \mathbb{R}^3 $ as a commutative group, and also 3-dimensional unit sphere $ \mathbb{S}^3 $ (or special unitary group $ \rm{SU}(2) $), 3-dimensional special orthogonal group $ \rm{SO}(3)$. Especially, $\mathbb{G}$ is a commutative group,  $ \mathbb{S}^3 $ or $ \rm{SO}(3)$ when $ \tau_\mathbb{G} = 0, 1$ or $\frac{1}{2} $, respectively \cite{esprito,ciftci}. Thus, we say that a three dimensional Lie group $\mathbb{G}$ with bi-invariant metric has a very rich structure and also, the following results are extended versions of the results which is given in \cite{deshmukh}.
\end{remark}

\section{Natural mates of Frenet curves in $ \mathbb{G} $}
In this section, we introduce natural mate of a Frenet curve in three dimensional Lie group $\mathbb{G}$ with bi-invariant metric. Also, we give some relationship between Frenet curve and its natural mates. Moreover, we obtain some results for the natural mate of a Frenet curve which is especially a general helix, a slant helix, a spherical curve or a curve with constant curvature. 

\begin{definition} \label{defn_3.1}
Let $\gamma:I\subset{\mathbb{R}} \to \mathbb{G}$ be a unit speed Frenet curve in $\mathbb{G}$ with the Frenet-Serret apparatus $ \{T,N,B,\kappa,\tau \}$. A unit speed curve $\beta:I\subset{\mathbb{R}} \to \mathbb{G}$ with Frenet-Serret apparatus $ \left\{ {\overline T ,\overline N ,\overline B ,\overline \kappa ,\overline \tau } \right\} $, is called the natural mate of the curve $\gamma$ if the curve $\beta$ is tangent to the principal normal vector of the curve $\gamma$ \mbox{(i.e. $\overline{T} = N$)}. 
\end{definition}

\begin{remark}\label{rmrk_3.1}
It is easily seen that the natural mate  curve of the Frenet curve $ \gamma $ is given by integral of principal normal $ N $  from Definition \ref{defn_3.1}. So, natural mate curve is same from algebraic viewpoint with principal-direction curve in \cite{kiziltug}. But natural mate curve is different from geometric viewpoint since it is defined as along the Frenet curve $ \gamma $ in three-dimensional Lie group $\mathbb{G}$.
\end{remark}

Now, let $ \gamma:I\subset{\mathbb{R}} \to \mathbb{G} $ be a unit speed Frenet curve in $\mathbb{G}$ with the Frenet-Serret apparatus $ \{T,N,B,\kappa,\tau \}$. Then, it is easily seen that $ \left\{ {N,\frac{{{\Omega^*}}}{\omega},\frac{\Omega}{\omega}} \right\} $ is an orthonormal basis in Lie algebra $ \mathfrak{g} $ of $ \mathbb{G} $ along the curve $\gamma$, where $ \Omega,\omega,\Omega^* $ are defined by (\ref{eq_2.14}), (\ref{eq_2.16}), (\ref{eq_2.17}), respectively. Now, by using (\ref{eq_2.10}), (\ref{eq_2.11}), (\ref{eq_2.16}), we get 
\begin{eqnarray*}
	{\left( {\frac{\kappa}{\omega}} \right)^\prime } =  - \frac{{\tau - {\tau_\mathbb{G}}}}{\sigma},~{\left( {\frac{{\tau  - {\tau _\mathbb{G}}}}{\omega}} \right)^\prime} = \frac{\kappa }{\sigma }.
\end{eqnarray*}
Also, by using equations (\ref{eq_2.9}) with these functions, we have
\begin{eqnarray}\label{eq_3.1}
\left[ {\begin{array}{*{20}{c}}
	N'\\
	{{{\left( {\frac{{{\Omega^*}}}{\omega}} \right)}^\prime }}\\
	{{{\left( {\frac{\Omega}{\omega}} \right)}^\prime }}
	\end{array}} \right] = \left[ {\begin{array}{*{20}{c}}
	0&\omega&0\\
	{ - \omega}&0&{\frac{\omega}{\sigma}}\\
	0&{ - \frac{\omega}{\sigma}}&0
	\end{array}} \right]\left[ {\begin{array}{*{20}{c}}
	N\\
	{\left( {\frac{{{\Omega^*}}}{\omega}} \right)}\\
	{\left( {\frac{\Omega}{\omega}} \right)}
	\end{array}} \right].
\end{eqnarray}
Moreover, equations (\ref{eq_3.1}) means that there exists a unit speed Frenet curve $ \beta $  in $\mathbb{G}$ with Frenet-Serret apparatus $ \{\overline{T},\overline{N},\overline{B},\overline{\kappa},\overline{\tau} \}$ which is satisfying the following equations 
\begin{eqnarray} \label{eq_3.2}
\left\{ \begin{array}{l}
\overline{T} = N,~\overline{N} = ({\Omega^*}/{\omega}),~\overline B = ({\Omega}/{\omega}),\\
\overline{\kappa} = {\omega}=\kappa \sqrt{1+H^2},\\
{ \overline{\tau} - \,\overline{\tau_\mathbb{G}}}= ({\omega}/{\sigma})={H'}/({1+H^2}),
\end{array} \right.  
\end{eqnarray}
where $\overline{\tau_\mathbb{G}} = \dfrac{1}{2}\left\langle [\overline{T},\overline{N}],\overline{B} \right\rangle$. Thus, we have $ \overline{T}=N $ along the curve $ \gamma $ by equations (\ref{eq_3.2}).
That is, the curve $ \beta $ is natural mate of the curve $ \gamma $ by Definition \ref{defn_3.1}. Then, we obtain the following theorem for natural mate curve of a Frenet curve in $ \mathbb{G} $.

\begin{theorem} \label{teo_3.1}
Let $\gamma$ be a unit speed Frenet curve in $\mathbb{G}$ with Frenet-Serret apparatus $ \{T,N,B,\kappa,\tau \}$. Then $\beta$ is natural mate of the curve $\gamma$ iff there exists a unit speed curve $\beta$ in $\mathbb{G}$ with Frenet-Serret apparatus $ \{\overline{T},\overline{N},\overline{B},\overline{\kappa},\overline{\tau} \}$ which is given by the equations  (\ref{eq_3.2}), where $ \Omega,\omega,\Omega^* $ are defined by (\ref{eq_2.14}), (\ref{eq_2.16}), (\ref{eq_2.17}), respectively and $\overline{\tau_\mathbb{G}} = \dfrac{1}{2}\left\langle [\overline{T},\overline{N}],\overline{B} \right\rangle$.
\end{theorem}
   
By using Theorem \ref{teo_3.1}, we get easily following corollaries for natural mate $\beta$ of a Frenet curve $\gamma$ which is a general helix or slant helix.   
   
\begin{corollary} \label{crl_3.1}
Let $\gamma$ be a unit speed Frenet curve in $\mathbb{G}$. Then $\gamma$ is a general helix if and only if the torsion $\overline{\tau}$ of natural mate $\beta$ of the curve $\gamma$ satisfies $\overline{\tau} - \overline{\tau_\mathbb{G}} = 0$ where $\overline{\tau_\mathbb{G}} = \dfrac{1}{2}\left\langle [\overline{T},\overline{N}],\overline{B} \right\rangle$.
\end{corollary}
 
\begin{corollary} \label{crl_3.2}
Let $\gamma$ be a unit speed Frenet curve in $\mathbb{G}$. Then $\gamma$ is a slant helix if and only if its natural mate $\beta$ is a general helix.	
\end{corollary}

\begin{remark}\label{rmrk_3.2}
We remark that Theorem \ref{teo_3.1}, Corollary \ref{crl_3.1} and Corollary \ref{crl_3.2} are same from algebraic viewpoint with Theorem 5, Theorem 7 and Theorem 8 in \cite{kiziltug}, respectively.
\end{remark}

Now, we give the following characterization for natural mate $\beta$ of the curve $\gamma$ which is a rectifying curve. 

\begin{corollary} \label{crl_3.3}
Let $\gamma$ be a unit speed Frenet curve in $\mathbb{G}$ with curvature $ \kappa $ and torsion $ \tau $. Then $\gamma$ is a rectifying curve iff the curvatures $\overline{\kappa}$ and $\overline{\tau}$ of its natural mate satisfy 
\begin{eqnarray} \label{eq_3.3}
{a}{{\kappa}^2} = (\overline{\tau} - \overline{\tau_\mathbb{G}}){\overline{\kappa}^2}
\end{eqnarray}
where nonzero certain constant $\lambda$  and $\overline{\tau_\mathbb{G}} = \dfrac{1}{2}\left\langle [\overline{T},\overline{N}],\overline{B} \right\rangle$.
\end{corollary}

\begin{proof}
We suppose that $ \gamma $ is rectifying curve in $ \mathbb{G} $ with arc-length parameter $ s $. Then, by Theorem \ref{teo_2.3}, the harmonic curvature of $ \gamma $ is given by $ H(s) = as + b $ where $ a \neq 0 $ and $ b $ are constants. Now, let the curvature and the torsion of the natural mate of $\gamma $ be $ \overline{{\kappa}} $ and $ \overline{{\tau}} $, respectively. Then, by the equations (\ref{eq_3.2}), 
\begin{eqnarray*}
	\overline \kappa &=& \kappa\sqrt {1 + {H^2}}  = \kappa\sqrt {1 + {{(as + b)}^2}}, \\
	\overline \tau  - \overline {{\tau_\mathbb{G}}} &=& \frac{{H'}}{{1 + {H^2}}} = \frac{a}{{1 + {{(as + b)}^2}}}.
\end{eqnarray*}
Hence, we get easily the equation (\ref{eq_3.3}) from the last equations.

On the other hand, let the curvatures of $ \gamma $ and its natural mate satisfy the the equation (\ref{eq_3.3}). Then, by using the equations (\ref{eq_3.2}), we obtain that $ H'(s)=a $ 
where $ a $ is non-zero constant and $ s $ is arc-length parameter of $ \gamma $. This means that  $ H'(s)=as+b $ and so $ \gamma $ is a rectifying curve in $ \mathbb{G} $ by Theorem \ref{teo_2.3}.
\end{proof}

Now, by using Theorem \ref{teo_3.1}, we give a result for natural mate $\beta$ of $\gamma$ which is especially a spherical curve.

\begin{corollary}\label{crl_3.4}
Let $\gamma$ be a unit speed Frenet curve in $\mathbb{G}$ with curvature $ \kappa $ and torsion $ \tau $. Then $\gamma$ is a spherical curve whose left shift is lying a sphere with radius $ r $ in  Lie algebra $\mathfrak{g}$ of $\mathbb{G} $ if and only if the curvature $\overline{\kappa}$ and torsion $\overline{\tau}$ of its natural mate satisfy 
\begin{eqnarray} \label{eq_3.4}
\frac{{\bar \kappa '}}{{\bar \kappa }} = \left( {\overline \tau   - \overline {{\tau _\mathbb{G}}} } \right)H \pm \left( {\tau  - {\tau _\mathbb{G}}} \right)\sqrt {{r^2}{\kappa ^2} - 1}.
\end{eqnarray}
\end{corollary}

\begin{proof}
We assume that $\gamma$ is a spherical curve whose left shift is lying a sphere with radius $ r $ in  Lie algebra $\mathfrak{g}$ of $\mathbb{G} $ with curvature $ \kappa $ and torsion $ \tau $. Then, 
\begin{eqnarray*}
\frac{{{{({\kappa ^\prime })}^2}}}{{{\kappa ^4}\left( {\tau  - {\tau _\mathbb{G}}} \right)}} + \frac{1}{{{\kappa ^2}}} = {r^2}
\end{eqnarray*}
and so
\begin{eqnarray}\label{eq_3.5}
{\kappa ^\prime } =  \pm \kappa \left( {\tau  - {\tau _\mathbb{G}}} \right)\sqrt {{r^2}{\kappa ^2} - 1}.
\end{eqnarray}
Also, by the equations (\ref{eq_3.2}) we get
\begin{eqnarray*}
\frac{{\bar \kappa'}}{{\bar \kappa}} = \frac{{\kappa'}}{\kappa} + \left( {\overline \tau  - \overline {{\tau_\mathbb{G}}} } \right)H
\end{eqnarray*}
and jointly with the equation (\ref{eq_3.5}), we obtain the equation (\ref{eq_3.4}).

On the other hand, let the curvatures of gamma and its natural mate satisfy the equation (\ref{eq_3.4}). After differentiation of $ {\bar \kappa} $ and with using $ {\overline \tau } $ in the equations (\ref{eq_3.2}), we have
\begin{eqnarray*}
\bar \kappa' = \kappa'\sqrt {1 + {H^2}}  + \kappa H\left( {\overline \tau  - \overline {{\tau_\mathbb{G}}} } \right)
\end{eqnarray*}
and from the last equation, we get
\begin{eqnarray} \label{eq_3.6}
\frac{{\kappa'}}{\kappa} = \frac{{\bar \kappa'}}{{\bar \kappa}} - H\left( {\overline \tau  - \overline {{\tau_\mathbb{G}}} } \right).
\end{eqnarray}
by using the equation (\ref{eq_3.4}) in the equation (\ref{eq_3.6}), we have
\begin{eqnarray} \label{eq_3.7}
\frac{{\kappa'}}{\kappa} =  \pm \left( {\tau - {\tau_\mathbb{G}}} \right)\sqrt {{r^2}{\kappa^2} - 1}. 
\end{eqnarray}
In that case, the equation (\ref{eq_3.7}) gives
\begin{eqnarray}\label{eq_3.8}
 \mp \frac{{\sqrt {{r^2}{\kappa^2} - 1} }}{\kappa} = {\left( {\frac{1}{\kappa}} \right)^\prime }\left( {\frac{1}{{\tau - {\tau_\mathbb{G}}}}} \right).
\end{eqnarray}
Then, after differentiation of  the equation (\ref{eq_3.8}) and bearing the equation (\ref{eq_3.7}) in mind, we obtain
\begin{eqnarray*}
{\left( {{{\left( {\frac{1}{\kappa}} \right)}^\prime }\left( {\frac{1}{{\tau - {\tau_\mathbb{G}}}}} \right)} \right)^\prime } =  - \frac{{\tau - {\tau_\mathbb{G}}}}{\kappa}.
\end{eqnarray*}
Thus, the last equation means that $ \gamma $ is a spherical curve in $ \mathbb{G} $ by the equation (\ref{eq_2.19}).
\end{proof}

\section{Spherical natural mates in $ \mathbb{G} $}

\begin{theorem}\label{teo_4.1}
Let $\gamma$ be a Frenet curve in $\mathbb{G}$ with constant curvature $\kappa = c >0$, then its natural mate $\beta$ is a spherical curve whose left shift is lying a sphere with radius $(1/c)$ in $\mathfrak{g}$. The converse holds if the torsion of the natural mate $\beta$ is not equal to the Lie group torsion of $ \beta $, that is $\overline{\tau} \neq {\overline{\tau_\mathbb{G}}}$.	
\end{theorem}

\begin{proof}
We suppose that $\gamma$ is a Frenet curve in $\mathbb{G}$ with constant curvature $\kappa = c >0$. Then, by using the the equations (\ref{eq_3.2}), the curvatures of its natural mate $ \beta $ are given by
\begin{eqnarray}\label{eq_4.1}
\bar \kappa = {\left( {{{\left( {\tau - {\tau_\mathbb{G}}} \right)}^2} + {c^2}} \right)^{1/2}}\,,\,\,\,\,\overline \tau  - \overline {{\tau_\mathbb{G}}}  = {{c{{\left( {\tau - {\tau_\mathbb{G}}} \right)}^\prime }} \mathord{\left/
		{\vphantom {{c{{\left( {\tau - {\tau_\mathbb{G}}} \right)}^\prime }} {\left( {{{\left( {\tau - {\tau_\mathbb{G}}} \right)}^2} + {c^2}} \right)}}} \right.
		\kern-\nulldelimiterspace} {\left( {{{\left( {\tau - {\tau_\mathbb{G}}} \right)}^2} + {c^2}} \right)}},
\end{eqnarray} 
where the harmonic curvature function of $ \gamma $ is $ H = \left( {\tau - {\tau_\mathbb{G}}} \right)/c $.

\noindent Case (1): If $ \tau = {\tau_\mathbb{G}} $, then we get
\begin{eqnarray*}
\bar \kappa = c\,,\,\,\,\,\overline \tau  - \overline {{\tau_\mathbb{G}}}  = 0.
\end{eqnarray*}
Hence, by Theorem \ref{teo_2.4}, the natural mate $ \beta $ is a spherical curve whose left shift is lying a sphere with radius $(1/c)$ in $\mathfrak{g}$.

\noindent Case (2): If $ \tau \neq {\tau_\mathbb{G}} $, then by using the equation (\ref{eq_4.1}), we get
\begin{eqnarray*}
{\left( {{{\left( {\frac{1}{{\bar \kappa}}} \right)}^\prime }\left( {\frac{1}{{\overline \tau  - \overline {{\tau_\mathbb{G}}} }}} \right)} \right)^\prime } + \overline H  =  - \frac{{c\,{{\left( {\tau - {\tau_\mathbb{G}}} \right)}^\prime }}}{{{{\left( {{{\left( {\tau - {\tau_\mathbb{G}}} \right)}^2} + {c^2}} \right)}^{3/2}}}} + \frac{{c\,{{\left( {\tau - {\tau_\mathbb{G}}} \right)}^\prime }}}{{{{\left( {{{\left( {\tau - {\tau_\mathbb{G}}} \right)}^2} + {c^2}} \right)}^{3/2}}}} = 0,
\end{eqnarray*}
and
\begin{eqnarray*}
{\left( {{{\left( {\frac{1}{{\bar \kappa}}} \right)}^\prime }\left( {\frac{1}{{\overline \tau  - \overline {{\tau_\mathbb{G}}} }}} \right)} \right)^2} + {\left( {\frac{1}{{\bar \kappa}}} \right)^2} = \frac{{{{\left( {\tau - {\tau_\mathbb{G}}} \right)}^2}}}{{{c^2}\left( {{{\left( {\tau - {\tau_\mathbb{G}}} \right)}^2} + {c^2}} \right)}} + \frac{1}{{{{\left( {\tau - {\tau_\mathbb{G}}} \right)}^2} + {c^2}}} = \frac{1}{{{c^2}}},
\end{eqnarray*}
where the harmonic curvature function of the natural mate $ \beta $ is $ \overline H  = \left( {\overline \tau  - \overline {{\tau_\mathbb{G}}} } \right)/\bar \kappa $. Hence, by Theorem \ref{teo_2.4}, the natural mate $ \beta $ is a spherical curve in $ \mathbb{G} $.

Conversely, under the case  $ \tau \neq {\tau_\mathbb{G}} $, we assume that the natural mate $ \beta $ is a spherical curve whose left shift is lying a sphere with radius $(1/c)$ in $\mathfrak{g}$. Then, by 
Theorem \ref{teo_2.4}, we have
\begin{eqnarray*}
\frac{{{{\left( {\bar \kappa'} \right)}^2}}}{{{{\left( {\overline \tau  - \overline {{\tau_\mathbb{G}}} } \right)}^2}{{\bar \kappa}^4}}} + \frac{1}{{{{\bar \kappa}^2}}} = \frac{1}{{{c^2}}},
\end{eqnarray*}
and so, we get
\begin{eqnarray*}
\frac{{\bar \kappa'}}{{\bar \kappa\sqrt {{{\bar \kappa}^2} - {c^2}} }} =  \pm \frac{{\overline \tau  - \overline {{\tau_\mathbb{G}}} }}{c}.
\end{eqnarray*}
If we integrate the last equation after put $ \bar \kappa = c\,u $, then we have
\begin{eqnarray*}
\int {\frac{{du}}{{c\,{u^2}\sqrt {1 - \left( {1/{u^2}} \right)} }} =  \pm \int {\frac{{\overline \tau  - \overline {{\tau_\mathbb{G}}} }}{c}ds} }, 
\end{eqnarray*}
and so,
\begin{eqnarray}\label{eq_4.2}
\bar \kappa = c\,\sec \left( {\int {\left( {\overline \tau  - \overline {{\tau_\mathbb{G}}} } \right)ds} } \right).
\end{eqnarray}
Now, by using the equations (\ref{eq_3.2}), the torsion $ \overline{{\tau}} $ of $ \beta $ satisfies $ \overline \tau  - \overline {{\tau_\mathbb{G}}}  = {{H'} \mathord{\left/{\vphantom {{H'} {\left( {1 + {H^2}} \right)}}} \right. \kern-\nulldelimiterspace} {\left( {1 + {H^2}} \right)}} $, where the harmonic curvature function of $ \gamma $ is $ H = \left( {\tau - {\tau_\mathbb{G}}} \right)/\kappa $. Moreover, if we take as $ H=\tan(\theta + b) $, where b is arbitrary constant, then we find $ \overline \tau  - \overline {{\tau_\mathbb{G}}}  = \theta' $. Hence, by using the equation (\ref{eq_4.2}), $ \bar \kappa = c\,\sec \left( {\theta + {\theta_0}} \right) $ such that $ \theta_0 $ is integration constant. Also, if we take as $ b = \theta_0 $, we have
\begin{eqnarray}\label{eq_4.3}
\bar \kappa = c\,\sec \left( {\theta + b} \right).
\end{eqnarray}
On the other hand, by using the equation (\ref{eq_3.2}), $ \bar \kappa = \kappa\sqrt {1 + {H^2}}$. Thus, for $ H=\tan(\theta + b) $, we obtain
\begin{eqnarray}\label{eq_4.4}
\bar \kappa = \kappa\sec \left( {\theta + b} \right).
\end{eqnarray}
Finally, by the equations (\ref{eq_4.3}) and  (\ref{eq_4.4}), it is easily seen that $ \kappa = c $.  
\end{proof}

\section{Natural mates with constant curvature in $ \mathbb{G} $}
Now, let the curvatures of a Frenet curve $ \gamma $ in $ \mathbb{G} $ be given by  
\begin{eqnarray}\label{eq_5.1}
\left\{ \begin{gathered}
\kappa = c\cos (\phi(s)), \hfill \\
\tau - {\tau_\mathbb{G}} = c\sin (\phi(s)), \hfill \\ 
\end{gathered}  \right.
\end{eqnarray}
where $ \phi(s) $ is a differentiable function and $ c $ is a non-zero positive constant. Then, by the equations (\ref{eq_2.16}) and (\ref{eq_3.2}), the curvature $ \overline{{\kappa}} $ of the natural mate $ \beta $ of $ \gamma $ is a positive constant. That is, $ \overline{{\kappa}} = c>0 $.

Now, we give the following characterization for natural mates with constant curvature.

\begin{theorem}\label{teo_5.1}
Let $\gamma$ be a Frenet curve in $\mathbb{G}$ and $ \beta $ be its natural mate. If the curvature $ \overline{{\kappa}} $ of the natural mate $ \beta $ is a positive constant $ c $, then the curvatures of $\gamma $ is given by  
\begin{eqnarray} \label{eq_5.2}
\left\{ \begin{gathered}
\kappa  = c\cos \left( {\int {\left( {\overline \tau   - \overline {{\tau _\mathbb{G}}} } \right)ds} } \right), \hfill \\
\tau  - {\tau _\mathbb{G}} = c\sin \left( {\int {\left( {\overline \tau   - \overline {{\tau _\mathbb{G}}} } \right)ds} } \right). \hfill \\ 
\end{gathered}  \right.
\end{eqnarray}
\end{theorem}

\begin{proof}
We suppose that  the curvature $ \overline{{\kappa}} $ of the natural mate $ \beta $ is a constant $ c>0 $. Then, by the equations (\ref{eq_3.2}), we have
\begin{eqnarray}\label{eq_5.3}
{\left( {\tau - {\tau_\mathbb{G}}} \right)^2} + {\kappa^2} = {c^2},
\end{eqnarray}
such that $ \overline{{\kappa}} = \omega$. Also, we get 
\begin{eqnarray}\label{eq_5.4}
\overline \tau  - \overline {{\tau_\mathbb{G}}}  = \frac{{{{\left( {{{\left( {\tau - {\tau_\mathbb{G}}} \right)} \mathord{\left/{\vphantom {{\left( {\tau - {\tau_\mathbb{G}}} \right)} c}} \right.\kern-\nulldelimiterspace} c}} \right)}^\prime }}}{{\sqrt {1 - {{\left( {{{\left( {\tau - {\tau_\mathbb{G}}} \right)} \mathord{\left/{\vphantom {{\left( {\tau - {\tau_\mathbb{G}}} \right)} c}} \right.\kern-\nulldelimiterspace} c}} \right)}^2}} }},
\end{eqnarray}
by the equations (\ref{eq_3.2}) jointly with (\ref{eq_5.3}). Moreover, we have $\tau - {\tau_\mathbb{G}} = c\sin \left( {\int {\left( {\overline \tau  - \overline {{\tau_\mathbb{G}}} } \right)ds} } \right) $ after integrating of the equation (\ref{eq_5.4}). Hence, by using the equation (\ref{eq_5.3}), it is easily seen that $ \kappa = c \cos\left({\int{\left( \overline{{\tau}}-\overline{{\tau _\mathbb{G}}}\right)ds}}\right) $.
\end{proof}

\begin{remark}\label{rmrk_5.1}
Theorem \ref{teo_5.1} means that the function $ \phi $ in (\ref{eq_5.1}), must be given by $ \phi(s)={\int {\left( {\overline \tau  - \overline {{\tau_\mathbb{G}}} } \right)ds} } $ when the curvature of the natural mate $\beta$ is a non-zero constant.
\end{remark}

Now, we give the following characterization for a spherical curve $ \gamma $ in $ \mathbb{G} $ whose the curvature of its natural mate $\beta$ is a non-zero constant.

\begin{theorem} \label{teo_5.2}
Let $\gamma$ be a Frenet curve in $\mathbb{G}$ with arc-length parameter $s$, and $ \beta $ be its natural mate whose the curvature $ \overline{{\kappa}} $ is a non-zero constant $ c $. Then $\gamma$ is a spherical curve in $ \mathbb{G} $ if and only if there is a positive constant $a \geq c$ such that the torsion $\overline{\tau}$ of $\beta$ satisfies 
	\begin{eqnarray} \label{eq_5.6}
		\overline{\tau} - \overline\tau_\mathbb{G} = \pm \dfrac{{c^{2}}\sqrt{a^{2}-c^{2}}{\cos(cs)}}{c^{2}+(a^{2}-c^{2}){\sin^{2}(cs)}}
	\end{eqnarray}
\end{theorem}

\begin{proof}
We suppose that $\gamma$ is a spherical curve in $ \mathbb{G} $ with the curvature $ \kappa $ and torsion $ \tau $, and also its natural mate is a Frenet curve with constant curvature $ \overline{{\kappa}}=c>0 $ and torsion $ \overline{{\tau}} $. Then, by (\ref{eq_2.18}) and  (\ref{eq_5.2}), there exist a positive constant $ r $ such that
\begin{eqnarray*}
\frac{1}{{{c^4}}}{\sec ^2}\left( {\int {\left( {\overline \tau  - \overline {{\tau_\mathbb{G}}} } \right)ds} } \right)\left( {{c^2} + {{\left( {\overline \tau  - \overline {{\tau_\mathbb{G}}} } \right)}^2}{{\sec }^2}\left( {\int {\left( {\overline \tau  - \overline {{\tau_\mathbb{G}}}} \right)ds} } \right)} \right) = {r^2}.
\end{eqnarray*}
Now, if we take as $ a = c^2 r $ in the previous equation, we have
\begin{eqnarray*}
{\left( {\overline \tau  - \overline {{\tau_\mathbb{G}}} } \right)^2}{\sec ^2}\left( {\int {\left( {\overline \tau  - \overline {{\tau_\mathbb{G}}} } \right)ds} } \right) + {c^2} = {a^2}{\cos ^2}\left( {\int {\left( {\overline \tau  - \overline {{\tau_\mathbb{G}}} } \right)ds} } \right),
\end{eqnarray*}
or equivalently,
\begin{eqnarray*}
{\left( {\overline \tau  - \overline {{\tau_\mathbb{G}}} } \right)^2}{\sec ^2}\left( {\int {\left( {\overline \tau  - \overline {{\tau_\mathbb{G}}} } \right)ds} } \right) = {a^2}{\cos ^2}\left( {\int {\left( {\overline \tau  - \overline {{\tau_\mathbb{G}}} } \right)ds} } \right) - {c^2}.
\end{eqnarray*}
The last equation implies $ a \geq c $, and so we get
\begin{eqnarray*}
\left( {\overline \tau  - \overline {{\tau_\mathbb{G}}} } \right)\sec \left( {\int {\left( {\overline \tau  - \overline {{\tau_\mathbb{G}}} } \right)ds} } \right) =  \pm \sqrt {{a^2}{{\cos }^2}\left( {\int {\left( {\overline \tau  - \overline {{\tau_\mathbb{G}}} } \right)ds} } \right) - {c^2}}  =  \pm \frac{{\sqrt {{a^2} - {c^2}{{\sec }^2}\left( {\int {\left( {\overline \tau  - \overline {{\tau_\mathbb{G}}} } \right)ds} } \right)} }}{{\sec \left( {\int {\left( {\overline \tau  - \overline {{\tau_\mathbb{G}}} } \right)ds} } \right)}},
\end{eqnarray*}
thus,
\begin{eqnarray*}
\frac{{c\left( {\overline \tau  - \overline {{\tau_\mathbb{G}}} } \right){{\sec }^2}\left( {\int {\left( {\overline \tau  - \overline {{\tau_\mathbb{G}}} } \right)ds} } \right)}}{{\sqrt {{a^2} - {c^2}{{\sec }^2}\left( {\int {\left( {\overline \tau  - \overline {{\tau_\mathbb{G}}} } \right)ds} } \right)} }} =  \pm c.
\end{eqnarray*}
After integrating by the last equation, we obtain
\begin{eqnarray*}
\arcsin \left( {\frac{c}{{\sqrt {{a^2} - {c^2}} }}\tan \left( {\int {\left( {\overline \tau  - \overline {{\tau_\mathbb{G}}} } \right)ds} } \right)} \right) =  \pm cs + {s_0},
\end{eqnarray*}
where $ s_0 $ is arbitrary constant. Thus, after applying a suitable translation with respect to $ s $, we have
\begin{eqnarray*}
\tan \left( {\int {\left( {\overline \tau  - \overline {{\tau_\mathbb{G}}} } \right)ds} } \right) =  \pm \frac{{\sqrt {{a^2} - {c^2}} }}{c}\sin (cs),
\end{eqnarray*}
or equivalently,
\begin{eqnarray*}
\int {\left( {\overline \tau  - \overline {{\tau_\mathbb{G}}} } \right)ds}  = \arctan \left( { \pm \frac{{\sqrt {{a^2} - {c^2}} }}{c}\sin (cs)} \right).
\end{eqnarray*}
Thus, after differentiating the previous equation, we obtain the equation $( \ref{eq_5.6}) $.

Conversely, let the curvature $ \overline{{\kappa}} $ of the natural mate $ \beta $ be a non-zero constant $ c $ and its torsion $ \overline{{\tau}} $ satisfy the equation (\ref{eq_5.6}). Then, by Theorem \ref{teo_5.1}, the curvatures of the curve $ \gamma $ are given by
\begin{eqnarray*}
\kappa = c\cos \left( {\int {\left( {\overline \tau  - \overline {{\tau_\mathbb{G}}} } \right)ds} } \right)\,,\,\,\,\,\tau - {\tau_\mathbb{G}} = c\sin \left( {\int {\left( {\overline \tau  - \overline {{\tau_\mathbb{G}}} } \right)ds} } \right).
\end{eqnarray*}
Thus, we have
\begin{eqnarray} \label{eq_5.7}
{\left( {\frac{1}{\kappa}} \right)^\prime }\left( {\frac{1}{{\tau - {\tau_\mathbb{G}}}}} \right) = \frac{{\left( {\overline \tau  - \overline {{\tau_\mathbb{G}}} } \right){{\sec }^2}\left( {\int {\left( {\overline \tau  - \overline {{\tau_\mathbb{G}}} } \right)ds} } \right)}}{{{c^2}}}\,,\,\,\,\,H = \tan \left( {\int {\left( {\overline \tau  - \overline {{\tau_\mathbb{G}}} } \right)ds} } \right),
\end{eqnarray}
where $ H $ is the harmonic curvature function of $ \gamma $. 
Moreover, by the equation (\ref{eq_5.6}), we get
\begin{eqnarray}\label{eq_5.8}
\tan \left( {\int {\left( {\overline \tau  - \overline {{\tau_\mathbb{G}}} } \right)ds} } \right) =  \pm \frac{{\sqrt {{a^2} - {c^2}} \sin (cs)}}{c},
\end{eqnarray} 
and so,
\begin{eqnarray}\label{eq_5.9}
{\sec ^2}\left( {\int {\left( {\overline \tau  - \overline {{\tau_\mathbb{G}}} } \right)ds} } \right) = \frac{{{c^2} + \left( {{a^2} - {c^2}} \right){{\sin }^2}(cs)}}{{{c^2}}}.
\end{eqnarray}
Now, by putting the equations (\ref{eq_5.6}) and (\ref{eq_5.9}) in the equation (\ref{eq_5.7}), we obtain
\begin{eqnarray*}
{\left( {\frac{1}{\kappa}} \right)^\prime }\left( {\frac{1}{{\tau - {\tau_\mathbb{G}}}}} \right) =  \pm \frac{{\sqrt {{a^2} - {c^2}} \cos (cs)}}{{{c^2}}}.
\end{eqnarray*}
By differentiating of the last equation and bearing the equations (\ref{eq_5.7}) and (\ref{eq_5.8}) in mind, we get
\begin{eqnarray*}
{\left( {{{\left( {\frac{1}{\kappa}} \right)}^\prime }\left( {\frac{1}{{\tau - {\tau_\mathbb{G}}}}} \right)} \right)^\prime } =  \mp \frac{{\sqrt {{a^2} - {c^2}} \sin (cs)}}{c} =  - H.
\end{eqnarray*}
Consequently, we obtain the equation (\ref{eq_2.19}). That is, $ \gamma $ is a spherical curve in $ \mathbb{G} $.
\end{proof}

Now, we give the following result which is obtained by Corollary \ref{crl_3.4} for a spherical curve in $ \mathbb{G} $ whose natural mate has non-zero constant curvature. 

\begin{corollary}\label{crl_5.2}
Let $\gamma$ be a Frenet curve in $\mathbb{G}$ with curvature $ \kappa $ and torsion $ \tau $, and also $ \beta $ be its natural mate with non-zero constant curvature $ \overline{{\kappa}} = c>0 $ and torsion $ \overline{{\tau}} $. Then $\gamma$ is a spherical curve whose left shift is lying a sphere with radius $(1/r)$ in $\mathfrak{g}$ if and only if $ \tau - {\tau_\mathbb{G}} = 0 $ or $ {\overline \tau  - \overline {{\tau_\mathbb{G}}} } =   \mp \kappa\sqrt {{r^2}{\kappa^2} - 1} $.
\end{corollary}

\begin{proof}
We assume that $ \gamma $ is a spherical curve whose left shift is lying a sphere with radius $(1/r)$ in $\mathfrak{g}$. Then, by Corollary \ref{crl_3.4}, the equation (\ref{eq_3.4}) is satisfied as a necessary and sufficient condition. Also, since the natural mate $ \beta $ of $ \gamma $ has non-zero constant curvature $ \overline{{\kappa}} = c>0 $, we obtain
\begin{eqnarray*}
\left( {\tau - {\tau_\mathbb{G}}} \right)\left( {\left( {\overline \tau  - \overline {{\tau_\mathbb{G}}} } \right) \pm \kappa\sqrt {{r^2}{\kappa^2} - 1} } \right) = 0.
\end{eqnarray*}
by using the equation (\ref{eq_3.4}). Consequently,  $ \gamma $ is a spherical curve in $\mathbb{G}$ iff either $ \tau = \tau_\mathbb{G} $ (i.e. $ \kappa = \frac{1}{r} $ ) or $ \overline \tau  - \overline {{\tau_\mathbb{G}}}  =  \mp \kappa\sqrt {{r^2}{\kappa^2} - 1}  $
\end{proof}

\section{Conjugate mates for Frenet curves in $ \mathbb{G} $}
In this section, we introduce conjugate mate of a Frenet curve in $ \mathbb{G} $. Moreover, we obtain corollaries for some special conjugate mate of a Frenet curve  in $ \mathbb{G} $. Especially, we give a characterization for the natural mate of a Frenet curve which is satisfying the condition $ {\tau-{\tau_\mathbb{G}}}=c \neq 0 $ where c is non-zero constant. 

\begin{definition} \label{defn_6.1}
	Let $\gamma:I\subset{\mathbb{R}} \to \mathbb{G}$ be a unit speed Frenet curve in $\mathbb{G}$ with the Frenet-Serret apparatus $ \{T,N,B,\kappa,\tau \}$. A unit speed curve ${\gamma^*}:I\subset{\mathbb{R}} \to \mathbb{G}$ with Frenet-Serret apparatus $ \left\{ {{T^*},{N^*},{B^*},{\kappa ^*},{\tau ^*}} \right\} $, is called the conjugate mate of the curve $\gamma$ if the curve ${\gamma^*}$ is tangent to the binormal vector of the curve $\gamma$ \mbox{(i.e. ${T}^*=B$)}.
\end{definition}

Let $ \gamma:I\subset{\mathbb{R}} \to \mathbb{G} $ be a unit speed Frenet curve in $\mathbb{G}$ with the Frenet-Serret apparatus $ \{T,N,B,\kappa,\tau \}$ and $ \tau_\mathbb{G}= \dfrac{1}{2}{\left\langle [T,N],B \right\rangle} $. Then, it is easily seen that $\left\{ {B,\, - {\text{sign}}(\tau  - {\tau _\mathbb{G}})N,\,{\text{sign}}(\tau  - {\tau _\mathbb{G}})T} \right\}$ is an orthonormal basis in Lie algebra $ \mathfrak{g} $ of $ \mathbb{G} $ along the curve $\gamma$ where $ {\tau-{\tau_\mathbb{G}}} \neq 0 $. By using equations (\ref{eq_2.9}), we have the following Frenet equations
\begin{eqnarray}\label{eq_6.1}
\left[ {\begin{array}{*{20}{c}}
	{B'} \\ 
	{ - {\text{sign}}(\tau  - {\tau _\mathbb{G}})N'} \\ 
	{{\text{sign}}(\tau  - {\tau _\mathbb{G}})T'} 
	\end{array}} \right] = \left[ {\begin{array}{*{20}{c}}
	0&{\left| {\tau  - {\tau _\mathbb{G}}} \right|}&0 \\ 
	{ - \left| {\tau  - {\tau _\mathbb{G}}} \right|}&0&\kappa  \\ 
	0&{ - \kappa }&0 
	\end{array}} \right]\left[ {\begin{array}{*{20}{c}}
	B \\ 
	{ - {\text{sign}}(\tau  - {\tau _\mathbb{G}})N} \\ 
	{{\text{sign}}(\tau  - {\tau _\mathbb{G}})T} 
	\end{array}} \right].
\end{eqnarray}
Under the condition $ {\tau-{\tau_\mathbb{G}}} \neq 0 $, the equations (\ref{eq_6.1}) means that there exists a unit speed Frenet curve $ {\gamma^*} $  in $\mathbb{G}$ with Frenet-Serret apparatus $ \left\{ {{T^*},{N^*},{B^*},{\kappa ^*},{\tau ^*}} \right\} $ which is satisfying the following equations 
\begin{eqnarray} \label{eq_6.2}
\left\{ {\begin{array}{*{20}{l}}
	{{T^*} = B,\;\,\,{N^*} =  - {\text{sign}}(\tau  - {\tau _\mathbb{G}})N,\;\,\,{B^*} = {\text{sign}}(\tau  - {\tau _\mathbb{G}})T,} \\ 
	{{\kappa ^*} = \left| {\tau  - {\tau _\mathbb{G}}} \right|,} \\ 
	{{\tau ^*} - {\tau _\mathbb{G}}^* = \kappa ,} 
	\end{array}} \right.
\end{eqnarray}
where ${\tau_\mathbb{G}}^* = \dfrac{1}{2}\left\langle [{T}^*,{N}^*],{B}^* \right\rangle$. Thus, we have ${T}^*=B $ along the curve $ \gamma $ by the equations (\ref{eq_6.2}).
That is, the curve $ {\gamma^*} $ is the conjugate mate of the curve $ \gamma $ by Definition \ref{defn_6.1}. Moreover, 
\begin{eqnarray*}
{\tau _\mathbb{G}}^* = \frac{1}{2}\left\langle {\left[ {B,\, - {\text{sign}}(\tau  - {\tau _\mathbb{G}})N} \right],\,{\text{sign}}(\tau  - {\tau _\mathbb{G}})T} \right\rangle  = {\tau _\mathbb{G}},
\end{eqnarray*}
and so,
\begin{eqnarray} \label{eq_6.3}
{\tau ^*} = \kappa  + {\tau _\mathbb{G}}.
\end{eqnarray}
Thus, we obtain the following theorem for conjugate mate of a Frenet curve in $ \mathbb{G} $.

\begin{theorem} \label{teo_6.1}
Let $\gamma$ be a unit speed Frenet curve in $\mathbb{G}$ with Frenet-Serret apparatus $ \{T,N,B,\kappa,\tau \}$ and $ \tau_\mathbb{G}= \dfrac{1}{2}{\left\langle [T,N],B \right\rangle} $. Then $\gamma^*$ is conjugate mate of the curve $\gamma$ iff there exists a unit speed curve $\gamma^*$ in $\mathbb{G}$ with Frenet-Serret apparatus $ \left\{ {{T^*},{N^*},{B^*},{\kappa ^*},{\tau ^*}} \right\} $ which is given by the equations (\ref{eq_6.2}) and (\ref{eq_6.3}), where $ {\tau-{\tau_\mathbb{G}}} \neq 0 $.
\end{theorem}

By using Theorem \ref{teo_6.1}, we get easily following results for conjugate mate ${\gamma^*}$ of a Frenet curve $\gamma$ which is a general helix or slant helix.   

\begin{corollary} \label{crl_6.1}
Let $\gamma$ be a unit speed Frenet curve in $\mathbb{G}$. Then $\gamma$ is a general helix if and only if the conjugate mate ${\gamma^*}$ of $\gamma$ is a general helix.
\end{corollary}

\begin{proof}
Let the harmonic curvatures of $ \gamma $ and $ \gamma^* $ be $ H $ and $ H^* $, respectively. Then, by using (\ref{eq_6.2}), we have
\begin{eqnarray} \label{eq_6.4}
{H^*} = \frac{{{\tau ^*} - {\tau _\mathbb{G}}^*}}{{{\kappa ^*}}} = \frac{\kappa }{{\left| {\tau  - {\tau _\mathbb{G}}} \right|}} = {\text{sign}}(\tau  - {\tau _\mathbb{G}})\frac{1}{H}.
\end{eqnarray}
Thus, the proof is clear.
\end{proof}

\begin{corollary} \label{crl_6.2}
Let $\gamma$ be a unit speed Frenet curve in $\mathbb{G}$. Then $\gamma$ is a slant helix if and only if the conjugate mate ${\gamma^*}$ of $\gamma$ is a slant helix.
\end{corollary}

\begin{proof}
Let harmonic curvatures of $ \gamma $ and $ \gamma^* $ be $ H $ and $ H^* $, respectively. Then, by using (\ref{eq_6.2}) and (\ref{eq_6.4}) , we get
\begin{eqnarray*}
{\sigma ^*} = \frac{{{\kappa ^*}{{\left( {{{\left( {{H^*}} \right)}^2} + 1} \right)}^{3/2}}}}{{{{\left( {{H^*}} \right)}^\prime }}} = \frac{{ - \kappa {{\left( {{H^2} + 1} \right)}^{3/2}}}}{{H'}} =  - \sigma.
\end{eqnarray*}
Hence, the proof is clear.
\end{proof}

Now, we give the following results by Theorem \ref{teo_3.1} and Theorem \ref{teo_6.1} for involute-evolute curves (see \cite{okuyucu}) and Bertrand curve couple (see \cite{okuyucu_2016})  . 

\begin{corollary}\label{crl_6.3}
Let $\gamma:I\subset{\mathbb{R}} \to \mathbb{G}$ be a unit speed Frenet curve with $ {\tau-{\tau_\mathbb{G}}} \neq 0 $. Then there exists a unique pair of a unit speed curve $\beta:I\subset{\mathbb{R}} \to \mathbb{G}$	and a Frenet curve ${\gamma}^*:I\subset{\mathbb{R}} \to \mathbb{G}$, such that the curves $\gamma$, $\beta$ and ${\gamma}^*$ are mutually orthogonal curves (i.e. involute-evolute curves).
\end{corollary}

\begin{corollary}\label{crl_6.4}
Let $\gamma$ be a unit speed Frenet curve and ${\gamma}^*$ be its conjugate mate in $\mathbb{G}$. Then  $\gamma$ and ${\gamma}^*$ are Bertrand curve couple in $\mathbb{G}$.
\end{corollary}

Now, we give the following characterization for the natural mate of a Frenet curve which is satisfying the condition $ {\tau-{\tau_\mathbb{G}}}=c \neq 0 $ where c is non-zero constant. 

\begin{theorem}\label{teo_6.2}
Let $\gamma$ be a unit speed Frenet curve with Frenet-Serret apparatus $ \{T,N,B,\kappa,\tau \}$ and $ \beta $ be its natural mate in $\mathbb{G}$. If ${\tau-{\tau_\mathbb{G}}}=c \neq 0 $ is a non-zero constant such that $ \tau_\mathbb{G}= \dfrac{1}{2}{\left\langle [T,N],B \right\rangle} $ then the natural mate $\beta$ is a spherical curve whose left shift is lying a sphere with radius $(1/c)$ in $\mathfrak{g}$.
\end{theorem}

\begin{proof}
Let $ \{T,N,B,\kappa,\tau \}$ be the Frenet-Serret apparatus of a unit speed Frenet curve $ \gamma $ and $\tau_\mathbb{G}$ be the its Lie torsion in $\mathbb{G} $. We suppose that ${\tau-{\tau_\mathbb{G}}}=c \neq 0 $ is a non-zero constant. By Theorem \ref{teo_6.1}, the conjugate mate ${\gamma}^*$ of the curve $ \gamma $ has the Frenet-Serret apparatus $ \left\{ {{T^*},{N^*},{B^*},{\kappa ^*},{\tau ^*}} \right\} $ which is satisfying by the equations (\ref{eq_6.2}) and (\ref{eq_6.3}) where ${\tau_\mathbb{G}}^* = \dfrac{1}{2}\left\langle [{T}^*,{N}^*],{B}^*\right\rangle$. 

Now, let $ \left\{ {\overline {{T}} ,\overline {{N}} ,\overline {{B}} ,\overline {{\kappa}} ,\overline {{\tau}} } \right\}$ be the Frenet-Serret apparatus of the natural mate $ \beta $ of $ \gamma $ and $ \left\{ {\overline {{T^*}} ,\overline {{N^*}} ,\overline {{B^*}} ,\overline {{\kappa^*}} ,\overline {{\tau^*}} } \right\} $ be the Frenet-Serret apparatus of the natural mate $ \overline{\gamma^*} $ of the conjugate mate $ \gamma^* $ with the equations (\ref{eq_3.2}). Then, by using Theorem \ref{teo_3.1} jointly with the equations (\ref{eq_2.16}), (\ref{eq_6.2}) and (\ref{eq_6.4}), we get
\begin{eqnarray}
\overline {{\kappa^*}}  &=& \sqrt {{{\left( {{\tau^*} - {\tau_\mathbb{G}}^*} \right)}^2} + {\kappa^*}}  = \sqrt {{\kappa^2} + {{\left( {\tau - {\tau_\mathbb{G}}} \right)}^2}}  = \overline \kappa \label{eq_6.5} \\
\overline {{\tau^*}}  - \overline {{\tau_\mathbb{G}}^*}  &=& \frac{{{{\left( {{H^*}} \right)}^\prime }}}{{1 + {{\left( {{H^*}} \right)}^2}}} =  - \sign\left( {\tau - {\tau_\mathbb{G}}} \right)\frac{{H'}}{{1 + {H^2}}} =  - \sign\left( {\tau - {\tau_\mathbb{G}}} \right)\left( {\overline \tau  - \overline {{\tau_\mathbb{G}}} } \right) \label{eq_6.6}
\end{eqnarray}
where $ \overline{\tau_\mathbb{G}} = \dfrac{1}{2}\left\langle [\overline{T},\overline{N}],\overline{B} \right\rangle $. Moreover, bearing Theorem \ref{teo_3.1} and Theorem \ref{crl_6.1} in mind, jointly with (\ref{eq_2.14}), (\ref{eq_2.16}) and (\ref{eq_2.17}), we obtain 
\begin{eqnarray}\label{eq_6.7}
\overline {{T^*}}  =  - \sign\left( {\tau  - {\tau _\mathbb{G}}} \right)\overline T \,,\,\,\,\,\overline {{N^*}}  =  - \sign\left( {\tau  - {\tau _\mathbb{G}}} \right)\overline N \,,\,\,\,\,\overline {{B^*}}  = \overline B.
\end{eqnarray} 
Thus, we conclude that the curve $ \beta $ and the curve $ \overline{\gamma^*} $ are congruent from the equations (\ref{eq_6.5}), (\ref{eq_6.6}) and (\ref{eq_6.7}). Also, $ {\tau  - {\tau _\mathbb{G}}} = c $ is non-zero constant from hypothesis, so the curvature $ {{\kappa ^*}}$ of the Frenet curve $ \overline{\gamma^*} $ is non-zero constant $ \left| c \right|  > 0 $ by using (\ref{eq_6.2}). Finally, if we apply Theorem \ref{teo_4.1} and take into account that the curves $ \beta $ and $ \overline{\gamma^*} $ are congruent, the proof is completed.
\end{proof}

Finally, when $ \mathbb{G} $ is Euclidean 3-space as a commutative group  (i.e. $ \tau_\mathbb{G} = 0 $), we give graphics of some special curves jointly with their natural mate and conjugate mate. In addition, we should also mentioned that the graphics are obtained by method of numerical solution in Mathematica.

\begin{figure}[H] 
	\centering
	\subfloat[]{%
		\includegraphics[width=0.31\textwidth,keepaspectratio]{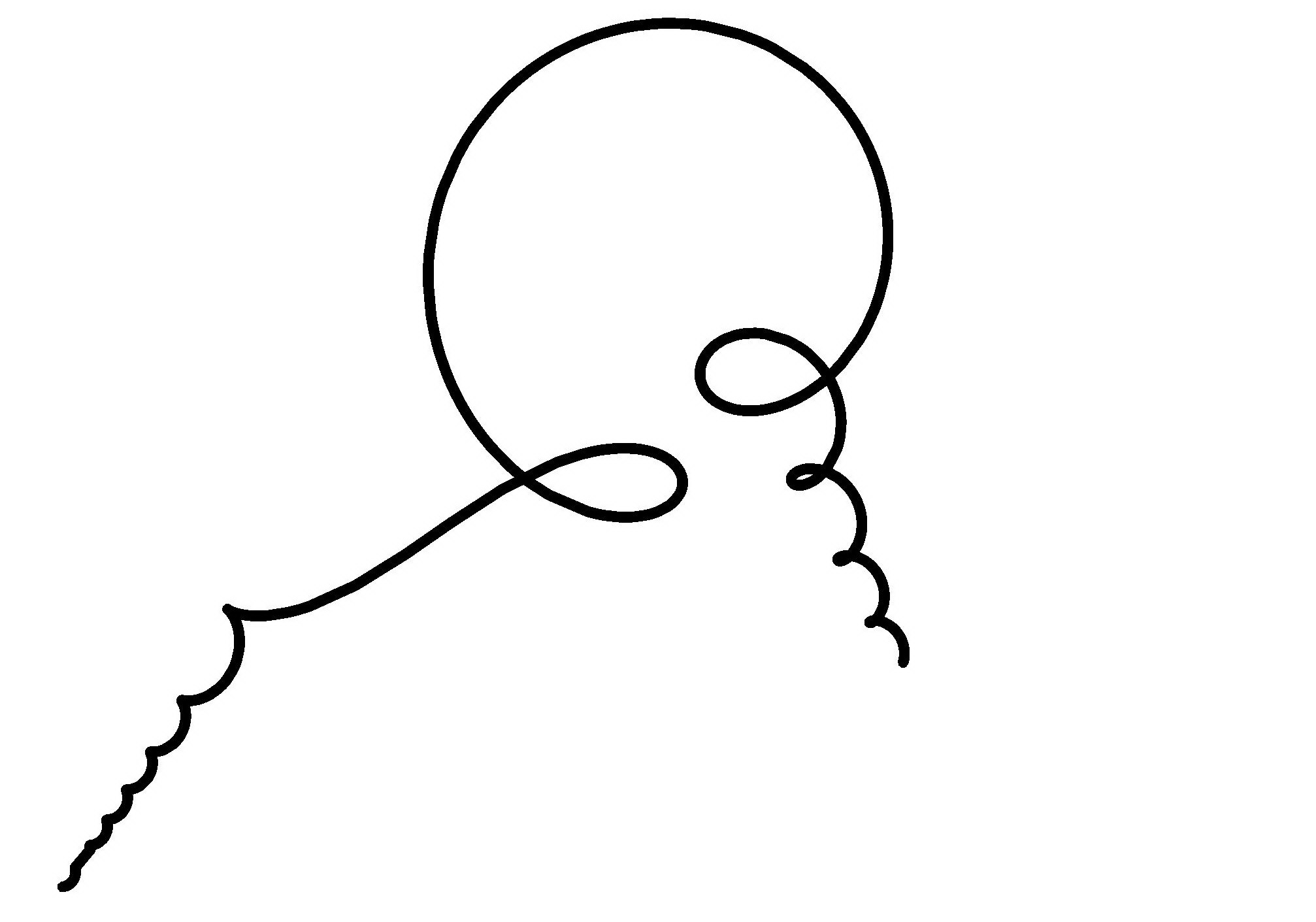}
		\label{fig1a}}
	\quad
	\subfloat[]{%
		\includegraphics[width=0.31\textwidth,keepaspectratio]{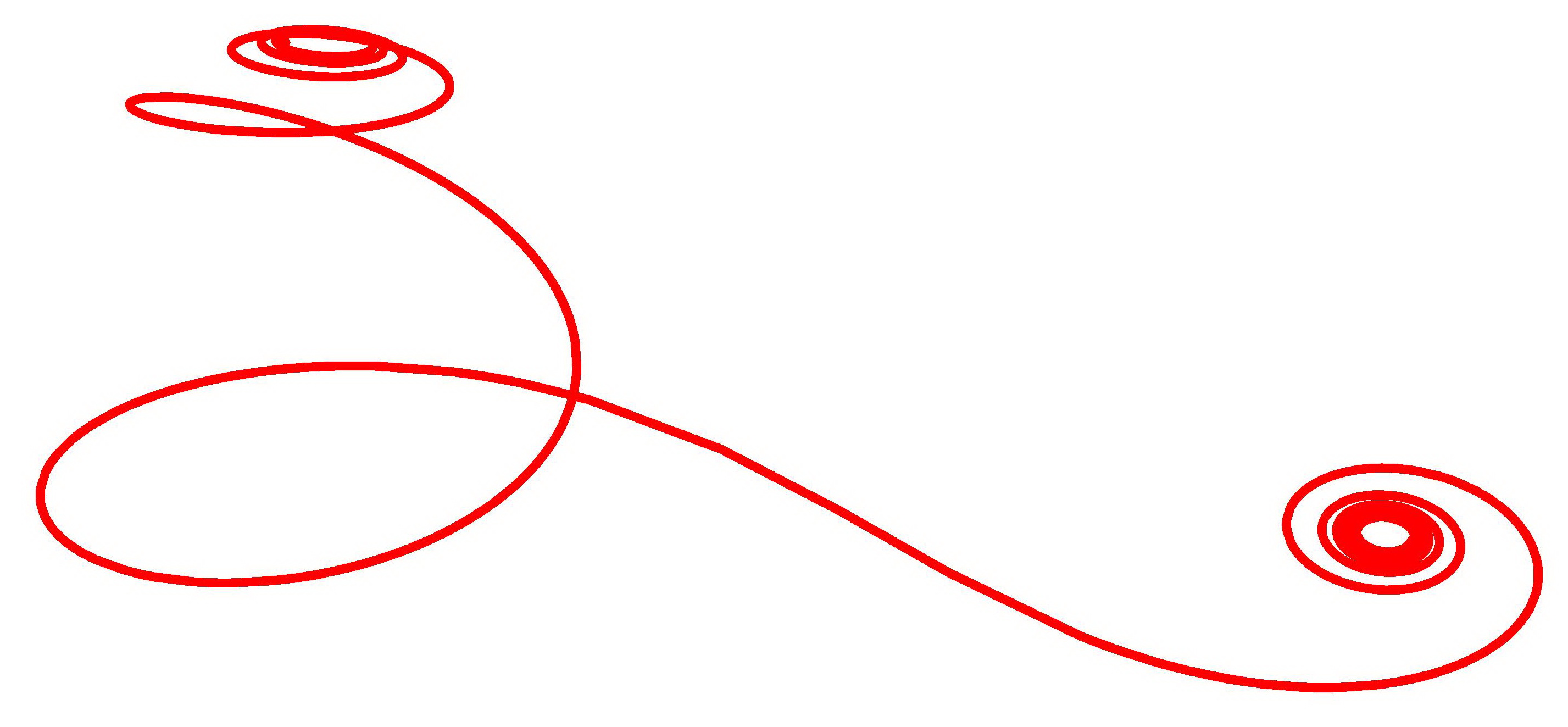}
		\label{fig1b}}
	\quad
	\subfloat[]{%
		\includegraphics[width=0.20\textwidth,keepaspectratio]{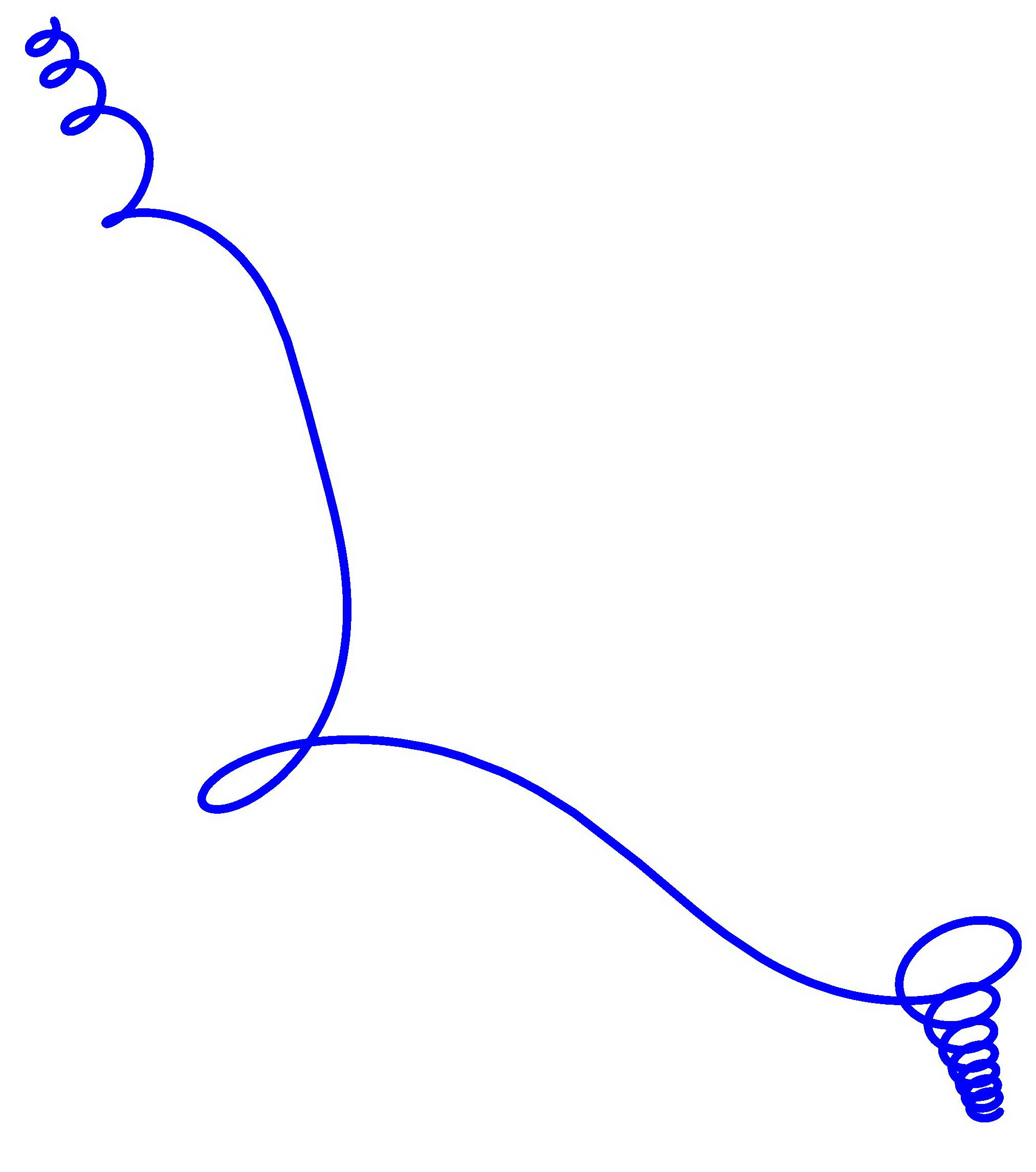}
		\label{fig1c}}
	\caption{\textbf{(a)} rectifying curve $ \gamma$ with curvature $ \kappa (s) = s - 1$ and torsion  $ \tau(s) = {s^2} + s - 2 $ ~\textbf{(b)} its natural mate $ \beta $ with $ \overline \kappa (s) = \sqrt {{{\left( {s - 1} \right)}^2}\left( {{s^2} + 4s + 5} \right)}  $ and $ \overline \tau (s) = 1/\left( {{s^2} + 4s + 5} \right)$~\textbf{(c)} its conjugate mate $ \gamma^*$ with $ {\kappa^*}(s) = \left| {{s^2} + s - 2} \right| $ and $ {\tau^*}(s) = s - 1 $.}
\end{figure}

\begin{figure}[H] 
	\centering
	\subfloat[]{%
		\includegraphics[width=0.13\textwidth,keepaspectratio]{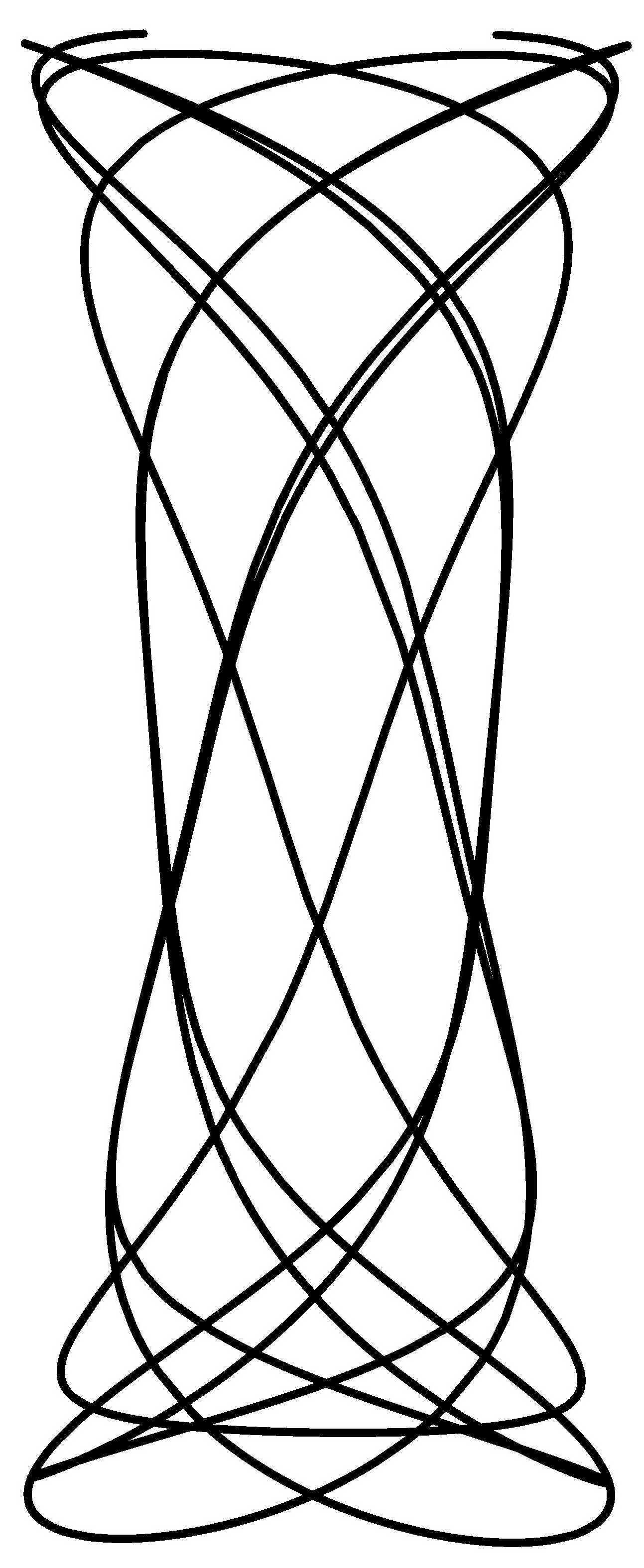}
		\label{fig2a}}
	\quad
	\subfloat[]{%
		\includegraphics[width=0.35\textwidth,keepaspectratio]{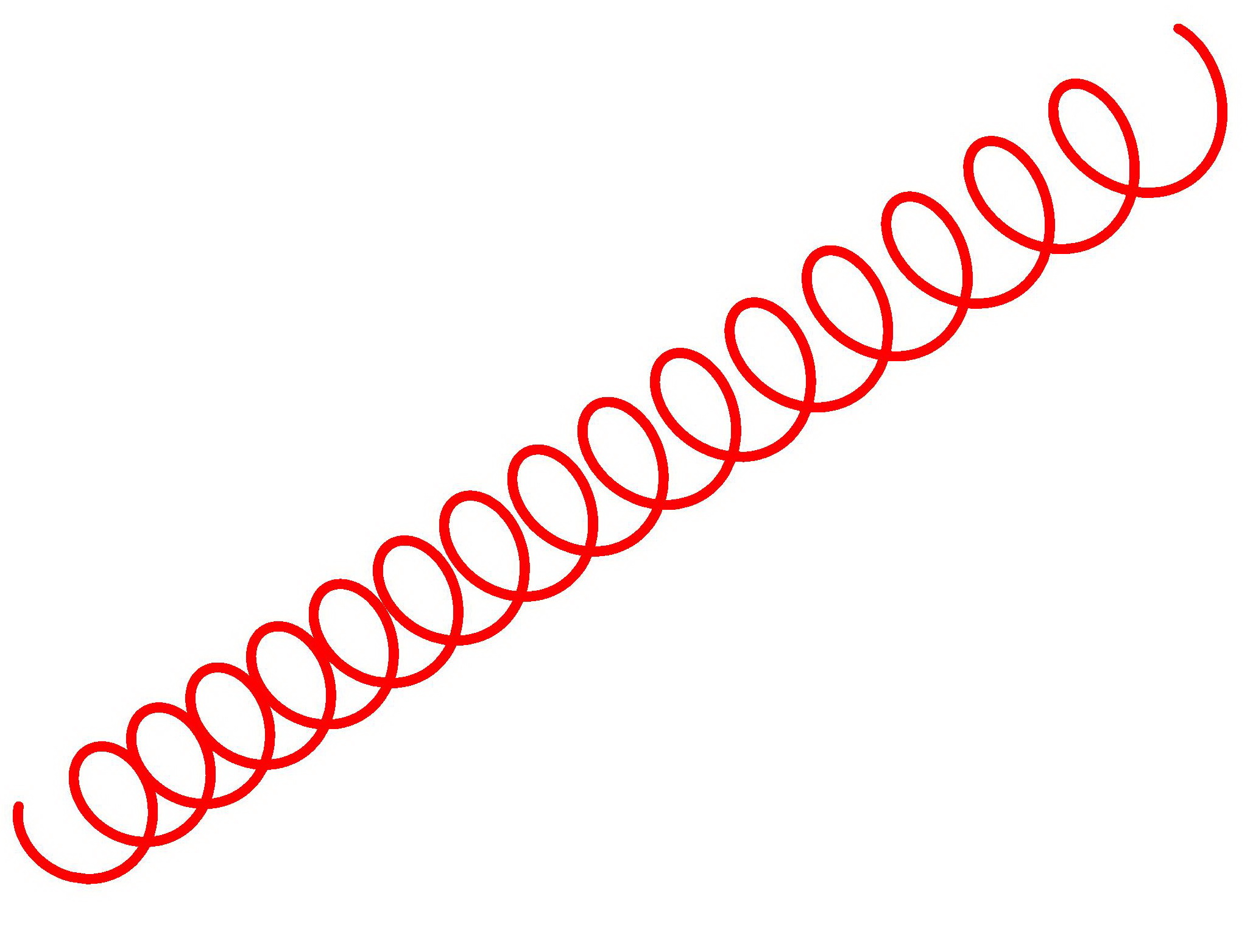}
		\label{fig2b}}
	\quad
	\subfloat[]{%
		\includegraphics[width=0.13\textwidth,keepaspectratio]{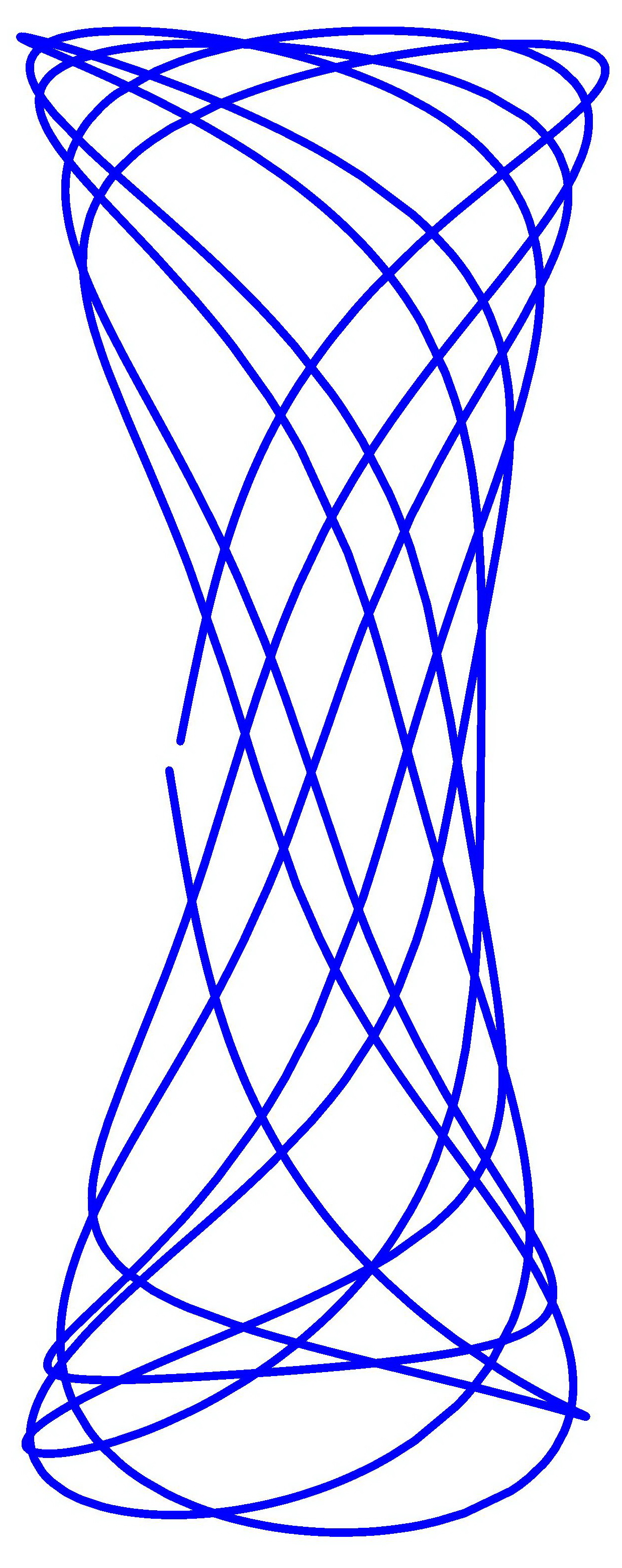}
		\label{fig2c}}
	\caption{\textbf{(a)} slant helix $ \gamma $ with $ \kappa (s) = 3 \cos(s)$ and $ \tau(s) = 3 \sin(s)$~\textbf{(b)} general helix natural mate $ \beta $ with $ \overline \kappa (s) = 3 $ and $ \overline \tau (s) = 1$~\textbf{(c)} slant helix conjugate mate $ \gamma^*$  with $ {\kappa^*}(s) = \left| 3 \sin(s) \right| $ and $ {\tau^*}(s) = 3 \cos(s) $.}
\end{figure}

\begin{figure}[H] 
	\centering
	\subfloat[]{%
		\includegraphics[width=0.26\textwidth,keepaspectratio]{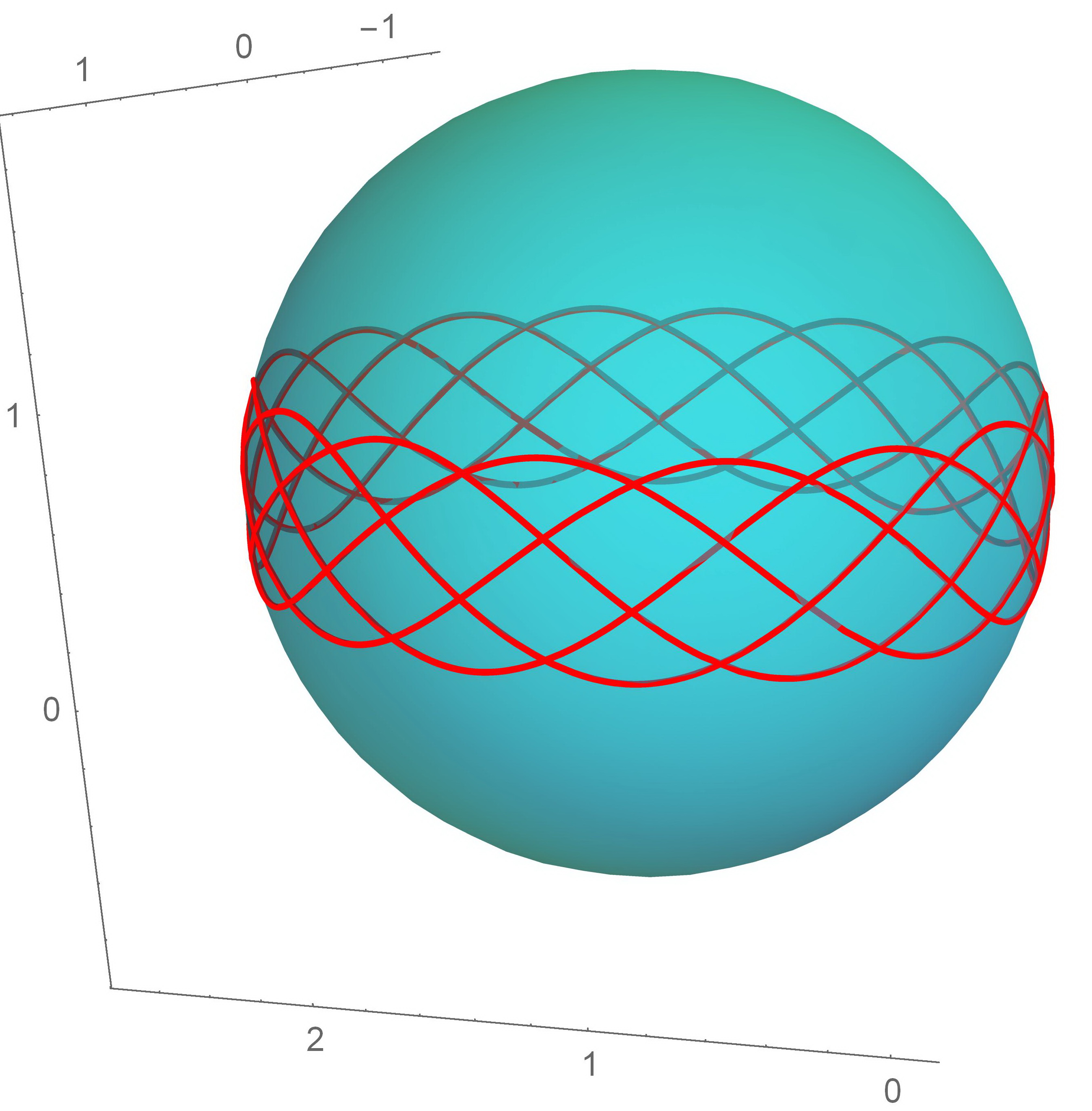}
		\label{fig3a}}
	\quad
	\subfloat[]{%
		\includegraphics[width=0.28\textwidth,keepaspectratio]{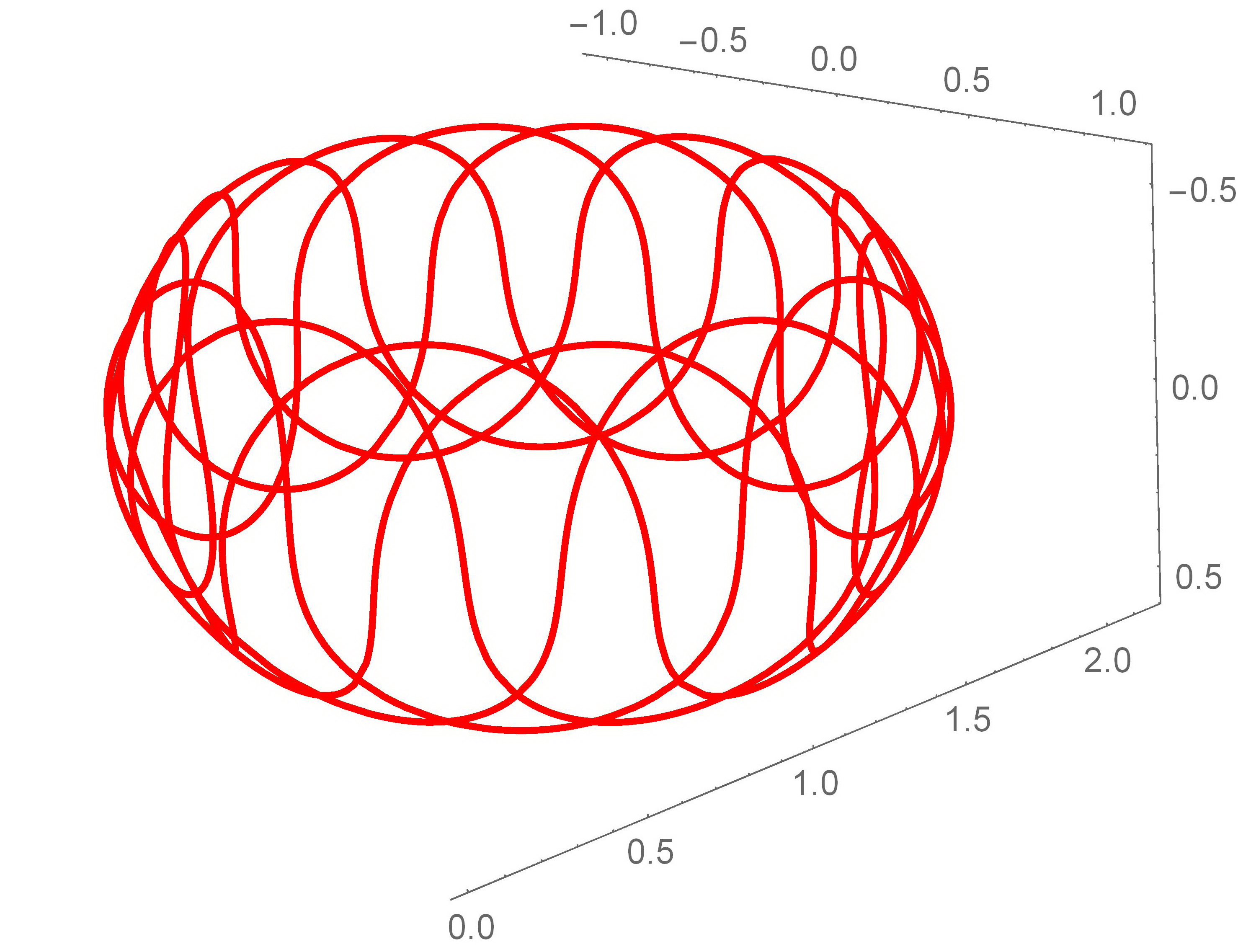}
		\label{fig3b}}
	\quad
	\subfloat[]{%
		\includegraphics[width=0.27\textwidth,keepaspectratio]{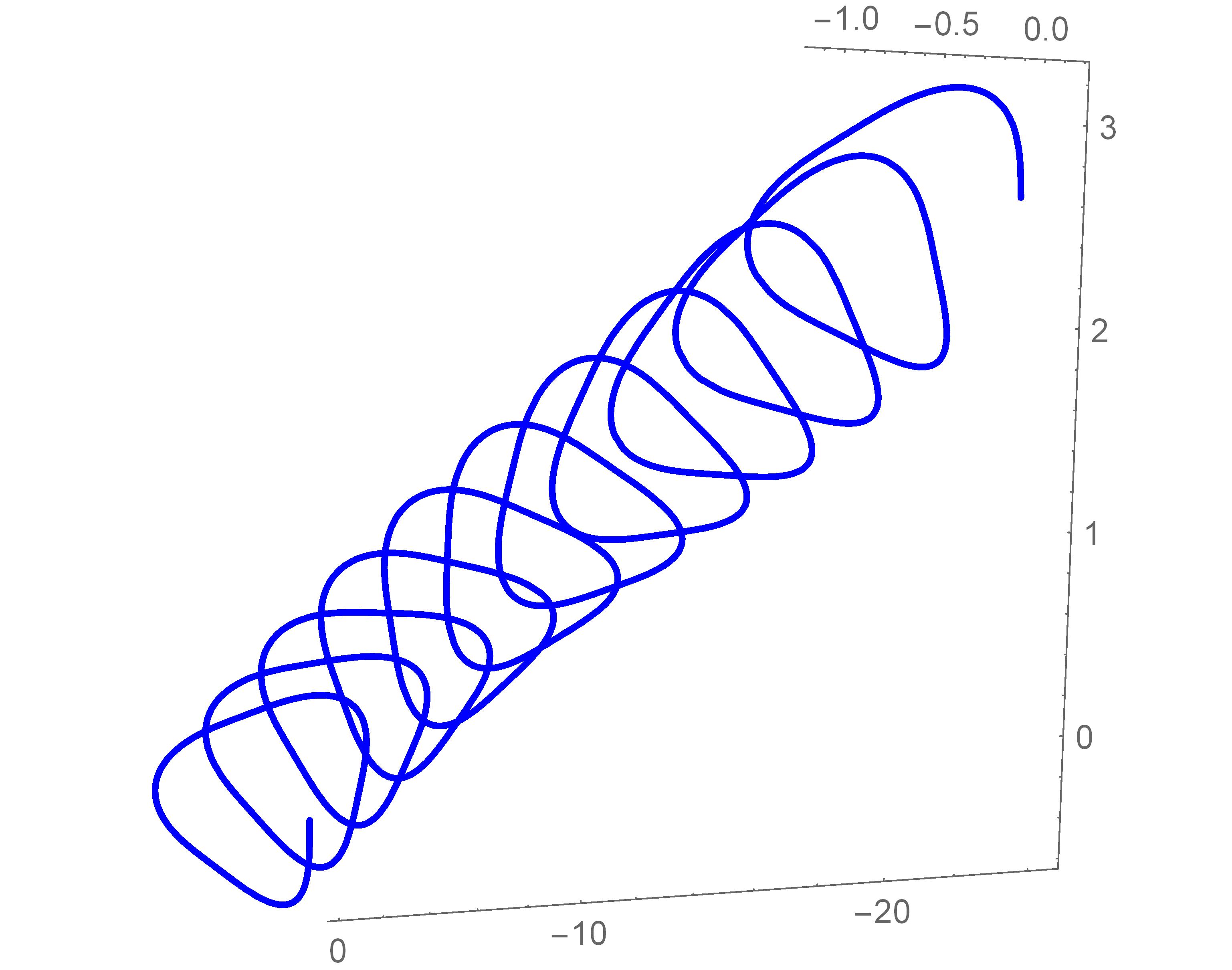}
		\label{fig3c}}
	\caption{\textbf{(a)} spherical curve $ \gamma$  with $ \kappa (s) = 2{\left( {1 + 7{{\sin }^2}(2s)} \right)^{-1/2}} $ and $ \tau(s) = 2\sqrt 7 \sin (2s){\left( {1 + 7{{\sin }^2}(2s)} \right)^{ - 1/2}}$ ~\textbf{(b)} Salkowski natural mate $ \beta $  with $ \overline \kappa (s) = 2 $ and $ \overline \tau (s) = \left( {4\sqrt 7 \cos (2s)} \right)/\left( {9 - 7\cos (4s)} \right) $~\textbf{(c)} conjugate mate $ \gamma^*$  with $ {\kappa^*}(s) = \left| {2\sqrt 7 \sin (2s){{\left( {1 + 7{{\sin }^2}(2s)} \right)}^{ - 1/2}}} \right|$ and $ {\tau^*}(s) = 2{\left( {1 + 7{{\sin }^2}(2s)} \right)^{ - 1/2}}$.}
\end{figure}

\begin{figure}[H] 
	\centering
	\subfloat[]{%
		\includegraphics[width=0.08\textwidth,keepaspectratio]{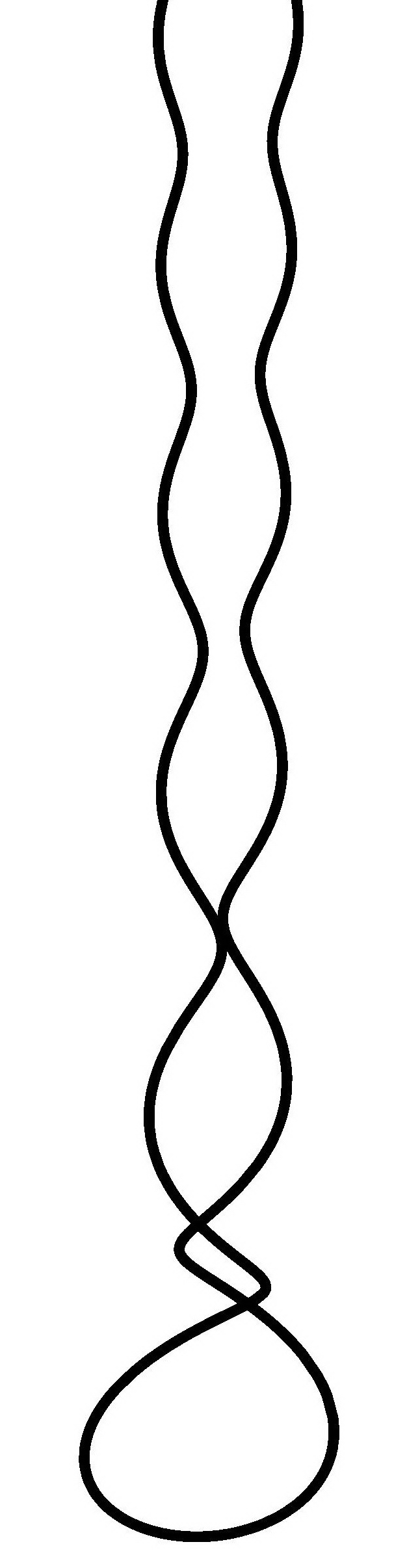}
		\label{fig4a}}
	\quad
	\subfloat[]{%
		\includegraphics[width=0.33\textwidth,keepaspectratio]{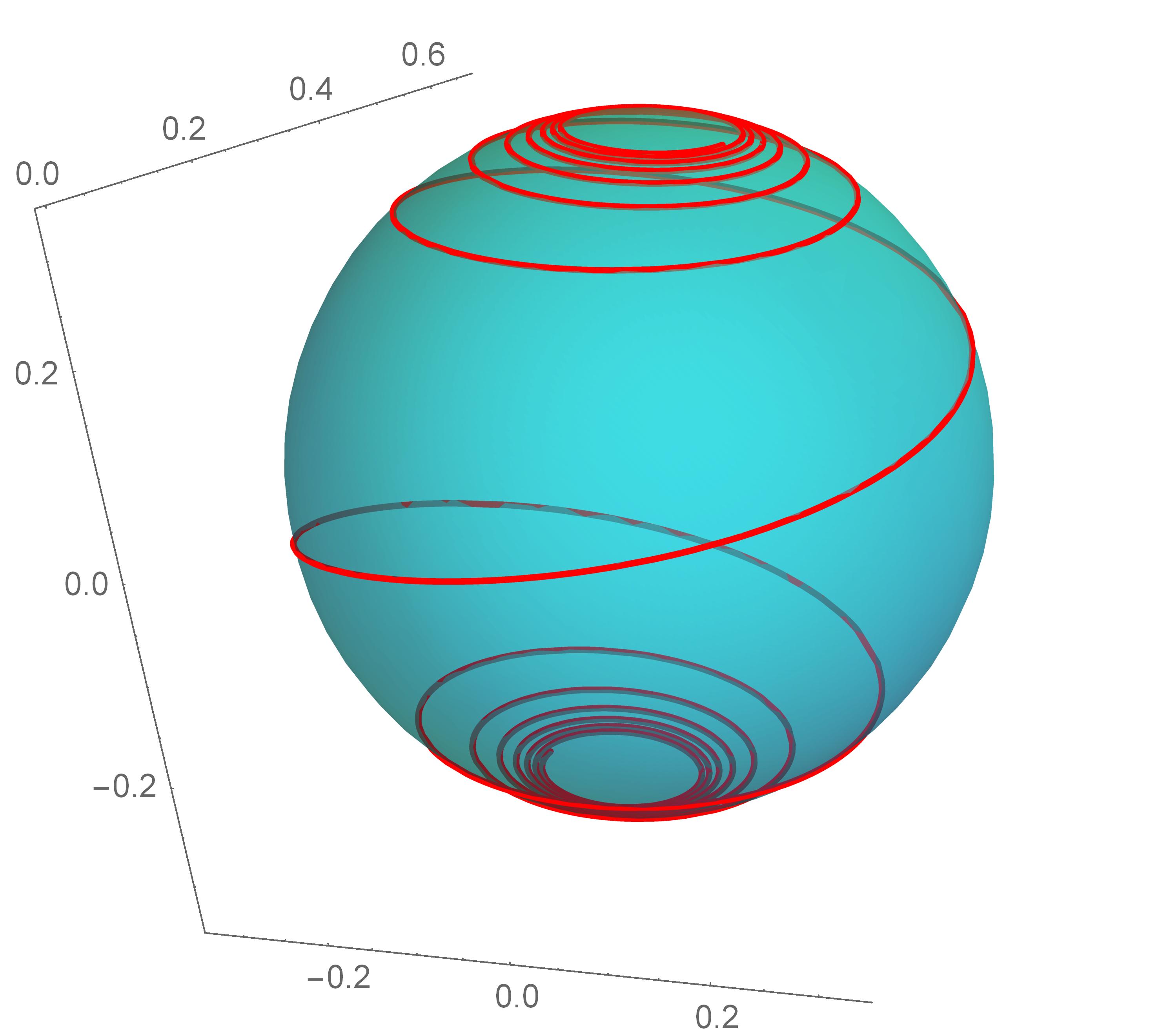}
		\label{fig4b}}
	\quad
	\subfloat[]{%
		\includegraphics[width=0.20\textwidth,keepaspectratio]{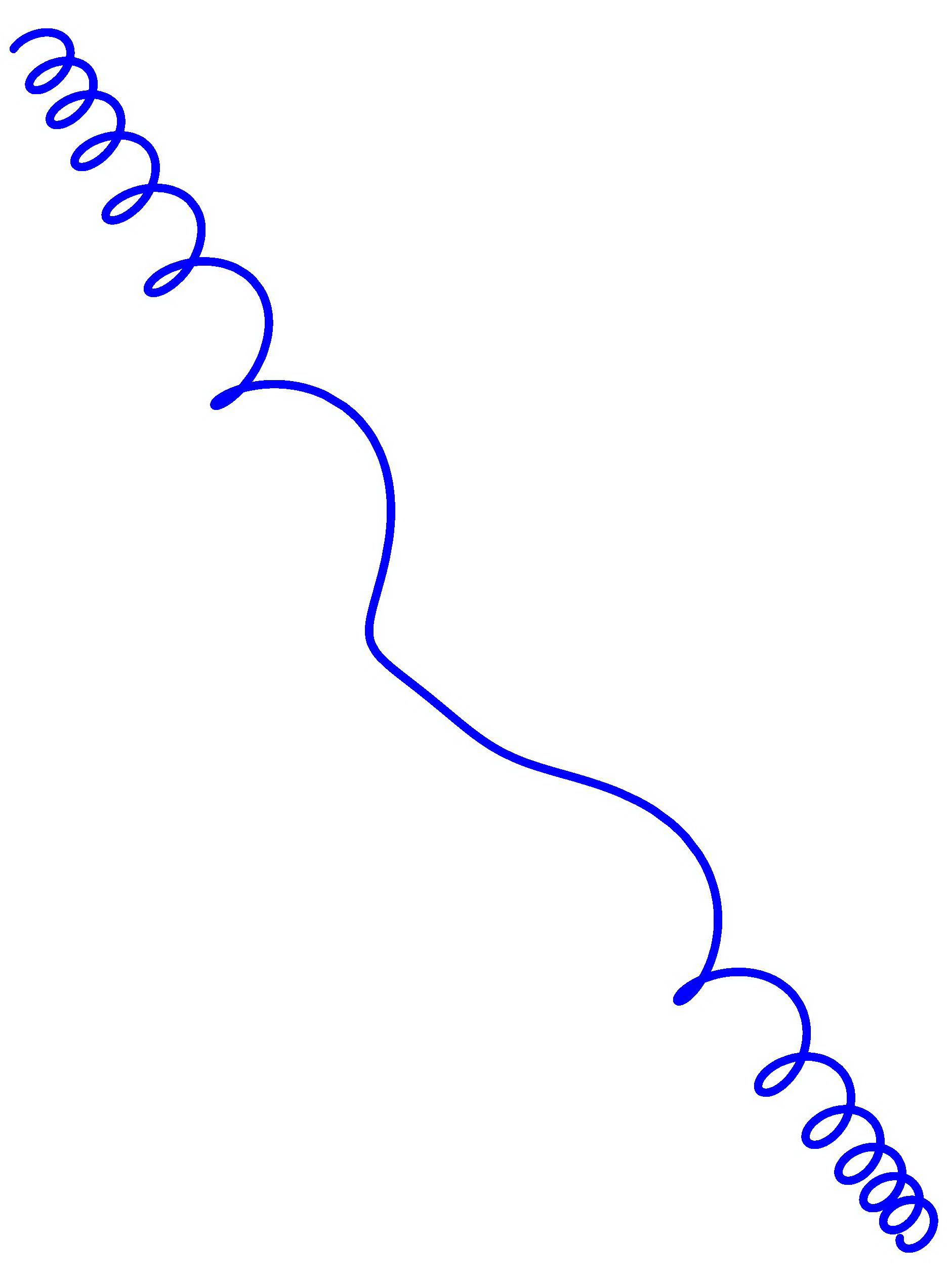}
		\label{fig4c}}
	\caption{\textbf{(a)} Salkowski curve $ \gamma$ with $ \kappa (s) = 3 $ and $ \tau(s) = 2s$ ~\textbf{(b)} spherical natural mate $ \beta $ with $ \overline \kappa (s) = \sqrt{9+4s^2} $ and $ \overline \tau (s) = {6}/{(9+4s^2)} $~\textbf{(c)} anti-Salkowski conjugate mate $ \gamma^*$ with $ {\kappa^*}(s) = \left| 2s \right| $ and $ {\tau^*}(s) = 3 $.}
\end{figure}

\begin{figure}[H] 
	\centering
	\subfloat[]{%
		\includegraphics[width=0.30\textwidth,keepaspectratio]{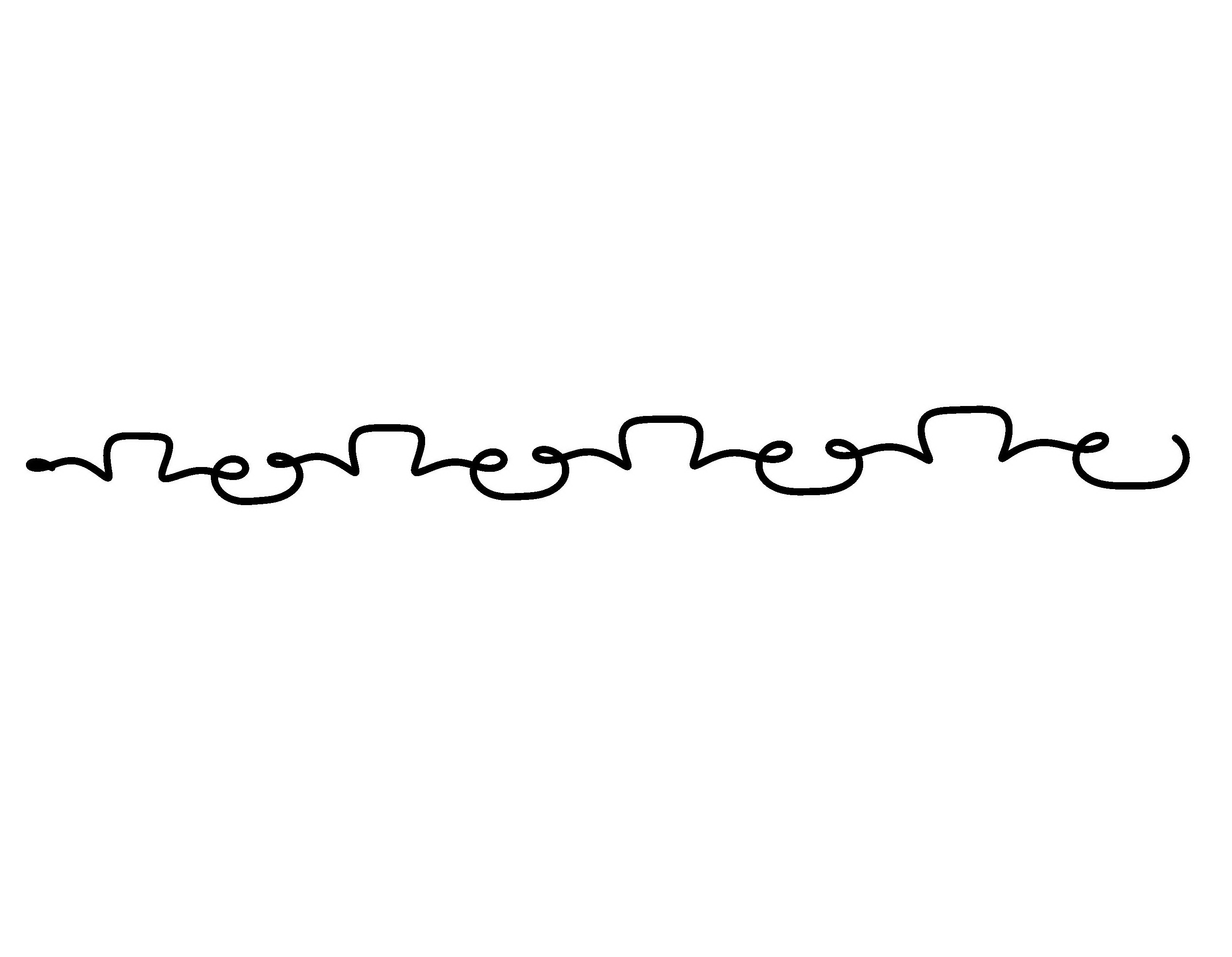}
		\label{fig5a}}
	\quad
	\subfloat[]{%
		\includegraphics[width=0.31\textwidth,keepaspectratio]{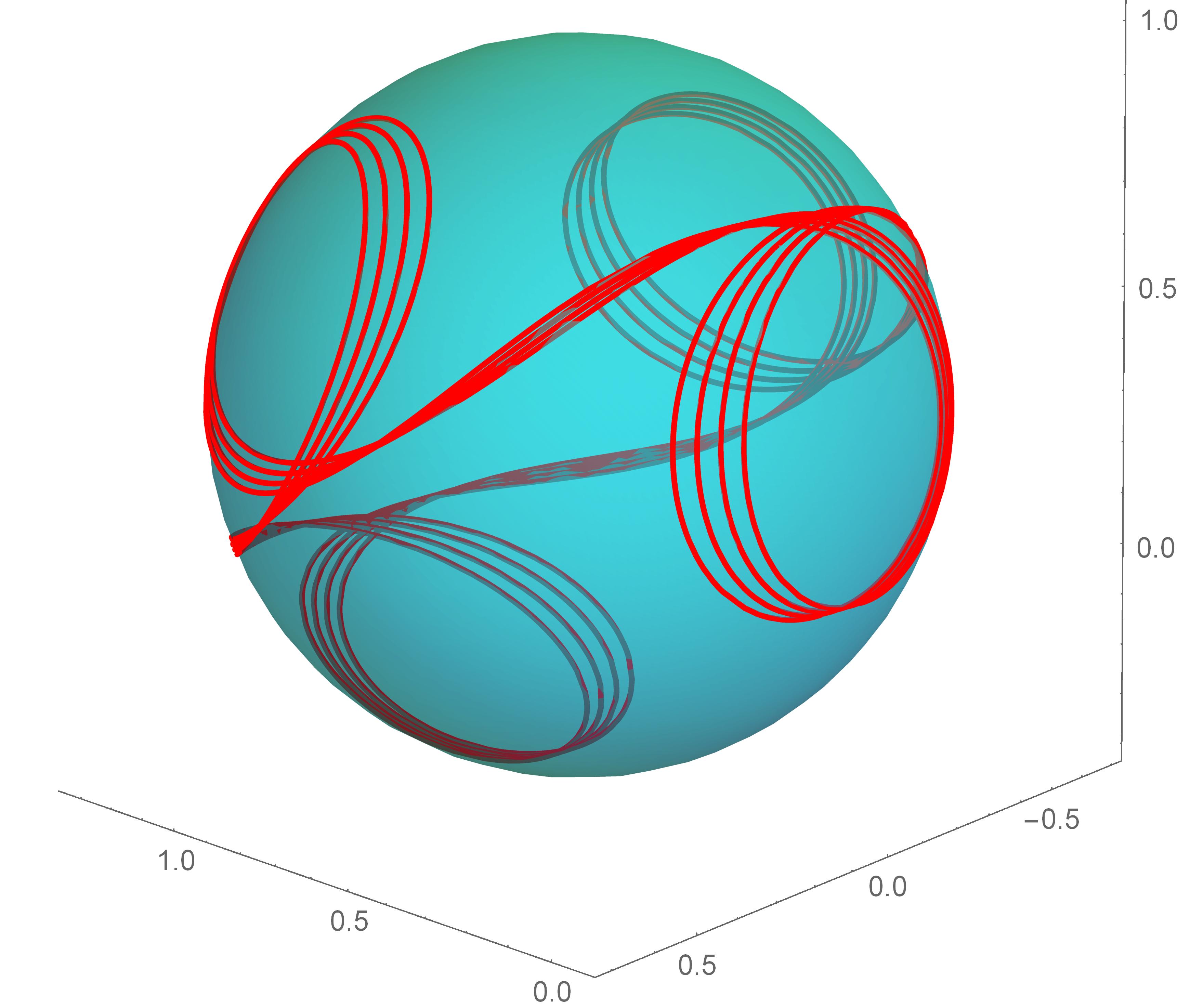}
		\label{fig5b}}
	\quad
	\subfloat[]{%
		\includegraphics[width=0.30\textwidth,keepaspectratio]{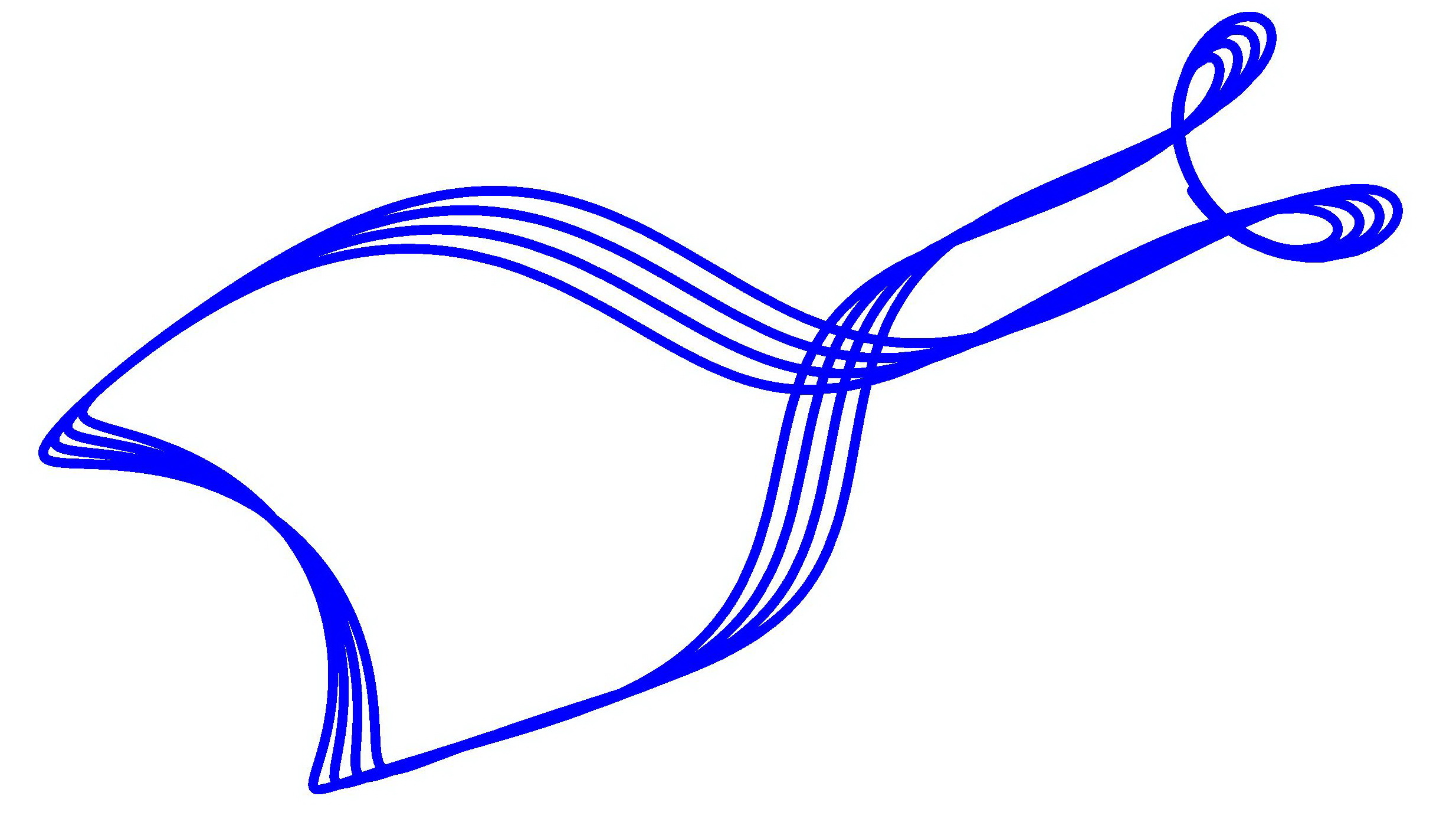}
		\label{fig5c}}
	\caption{\textbf{(a)} anti-Salkowski curve $ \gamma $ with $ \kappa (s) = 3 \cos(s)  $ and $ \tau(s) = \sqrt{2} $~\textbf{(b)} spherical natural mate $ \beta $ with $ \overline \kappa (s) = \sqrt{2+9s^2}$ and $ \overline \tau (s) = {3\sqrt{2}\sin(s)}/({2+9{\cos^2(s)}}) $~\textbf{(c)} Salkowski conjugate mate $ \gamma^*$ with $ {\kappa^*}(s) = \sqrt{2} $ and $ {\tau^*}(s) = 3 \cos(s)$.}
\end{figure}


\end{document}